\documentclass[12pt,leqno,fleqn]{amsart}  
\usepackage[color=green!15,bordercolor=red,linecolor=red!30]{todonotes}
\usepackage{amsmath,amstext,amsthm,amssymb,amsxtra}
\usepackage[normalem]{ulem}
\usepackage{txfonts} % pxfonts txfonts 
\usepackage[T1]{fontenc}
\usepackage{lmodern}

 \usepackage{euler}   % better than the option below

\usepackage{tikz}

\usepackage{mathtools}
\mathtoolsset{showonlyrefs,showmanualtags}

\usepackage{hyperref} %,hypertexnames=false,colorlinks,[pagebackref]
\hypersetup{
    colorlinks=true,       % false: boxed links; true: colored links
    linkcolor=blue,          % color of internal links
    citecolor=magenta,        % color of links to bibliography
    filecolor=magenta,      % color of file links
    urlcolor=cyan           % color of external links
}

\usepackage{amsrefs}  

\allowdisplaybreaks
\theoremstyle{plain} % definition 
\newtheorem{lemma}[equation]{Lemma} 
\newtheorem{proposition}[equation]{Proposition} 
\newtheorem{theorem}[equation]{Theorem}

\theoremstyle{definition}
\newtheorem{definition}[equation]{Definition} 

\theoremstyle{remark}
\newtheorem{remark}[equation]{Remark}

\numberwithin{equation}{section}

\title[Local $ Tb$  Theorem: A Direct Proof]{On the  Local $ Tb$ Theorem: A Direct Proof under the Duality Assumption} 
\subjclass[2000]{Primary: 42B20 Secondary: 42B25, 42B35}
\keywords{Local Tb theorem, T1 theorem, corona, twisted martingale transform, stopping cubes}
\author{Michael T. Lacey}   %  can use \and  

\address[M.T.L.]{ School of Mathematics, Georgia Institute of Technology, Atlanta GA 30332, USA}
\email {lacey@math.gatech.edu}
\thanks{Research supported in part by grant NSF-DMS 0968499, 
and  a grant from the Simons Foundation (\#229596 to Michael Lacey).}

\author{Antti V. V\"ah\"akangas}
\address[A.V.V.]{Department of Mathematics and Statistics,
P.O. Box 68, FI-00014 University of Helsinki, Finland} \email{antti.vahakangas@helsinki.fi} \thanks{A.V.V.  was
supported by the School of Mathematics, Georgia Institute of Technology, and 
by the Finnish Academy of Science and Letters, Vilho, Yrj\"o and
Kalle V\"ais\"al\"a Foundation.
}

\begin{document}

\begin{abstract}
	We give a new direct proof of the `local $ Tb$ Theorem in the Euclidean setting and under the assumption of dual exponents'. 
	This theorem provides a flexible framework for proving the boundedness of a Calder\'on--Zygmund operator, 
	supposing the existence of systems of local accretive functions.  We assume that the integrability exponents
	on these systems of functions  are of the form $ 1/p+1/q\le1$, 
	the `dual case' $1/p+1/q=1$ being the most difficult one. 
	Our proof is direct: it avoids a reduction to the perfect dyadic case unlike some previous approaches.
	 The principal point of interest  is in the 
	use of random grids and the corresponding construction of the corona.  We also utilize   
	certain twisted martingale transform inequalities.  
\end{abstract}	
	
\maketitle

%\setcounter{tocdepth}{1}
%\tableofcontents 

%\todo[inline]{} 
%%  ENUMERATE
%\begin{enumerate}
%\end{enumerate}
%% ENUMERATE

%\todo[inline] {}

%%%%%%%%%%%%%%%%%%%%%%%%%%%%%% SECTION  SECTION SECTION
%%%%%%%%%%%%%%%%%%%%%%%%%%%%%% SECTION  SECTION SECTION 
\section{Introduction} %\label{s.}

Our subject is the local $ Tb$ theorem in the classical Euclidean setting.
There are many results under this topic, all of which extend the 
David--Journ\'e $ T1$ Theorem \cite{MR763911}, and the $ Tb$ Theorem of Christ \cite{MR1096400}, by  giving 
 flexible conditions under which an operator $ T$ with a Calder\'on--Zygmund kernel 
 extends to a bounded linear operator on $ L ^2 $; 
 the lectures of Hofmann \cite{MR2664559}  indicate the range of interests in this type of results. 
  By `local' we understand that the $Tb$ conditions involve a family of test functions $b_Q$, one for each cube $Q$, which 
should satisfy a non-degeneracy condition on its `own' $Q$.
Furthermore, both $b_Q$ and $Tb_Q$ are subject to
normalized  integrability conditions on $Q$. 
Symmetric assumptions are imposed on $T^*$.

 The goal of this paper is to  give a new
 direct proof of a known local $Tb$ theorem, Theorem \ref{t.main1}. This theorem applies, in particular, when the integrability conditions imposed in the hypotheses 
are those in duality, namely, $ 1/{p_1}+1/{p_2}= 1$.
%We do this in the setting of Lebesgue measure.  
Our argument is direct in the sense that it avoids 
a reduction to the so-called perfect  dyadic  case, as in Auscher-Yang \cite{MR2474120}.
%\cite{MR1934198}. Other direct proofs, avoiding this reduction, are available. We will describe
%them later.
A companion paper \cite{lv-perfect} addresses a perfect dyadic variant of Theorem \ref{t.main1}
for the full range $1<p_1,p_2<\infty$; it contains many of the features of the argument in the present paper, with significantly fewer technicalities.

We say that $T$ is a \emph{Calder\'on--Zygmund operator}, if
it is a bounded linear operator on $L^2(\mathbf{R}^n)$ with
the following representation: for every $f\in L^2(\mathbf{R}^n)$,
 \begin{equation*}
Tf(x) = \int_{\mathbf{R}^n} K (x,y) f (y)\,dy\,,\qquad x\not\in \mathrm{supp}(f)\,,
\end{equation*}
where the kernel 
$ K : \mathbf{R}^n \times \mathbf{R}^n \to \mathbf{C}$ is assumed to satisfy the following estimates for some $\eta>0$:
\begin{align} \label{e.size}
	\lvert  K (x,y)\rvert & \leq \lvert  x-y\rvert ^{-n} \,, \qquad  x\not= y \,, 
	\\ \label{e.smoothness}
	\lvert  K (x,y)- K (x',y)\rvert
	+	\lvert  K (y,x)- K (y,x')\rvert
	& \leq \frac { \lvert  x-x'\rvert ^{\eta } } {\lvert  x-y\rvert ^{n+ \eta } }\,, 
	\qquad \lvert  x-x'\rvert< \tfrac 12 \lvert  x-y\rvert\,.
\end{align}
We define $\mathbf{T}$ to be the norm of $ T$ as an operator on $ L ^2(\mathbf{R}^n) $.  

%%%%%%%%%%%%%%%%%%%%%%%%%%%%%%  DEFINITION DEFINITION DEFINITION
\begin{definition}\label{d.system}
	Fix $ 1<p < \infty $.   
	A collection of functions $ \{b _Q \;:\; Q\subset \mathbb R ^{n} \textup{ is a cube}\} $ is 
	called \emph{a system of $p$-accretive functions with constant $\mathbf A>1$} if the
	following conditions (1) and (2) hold for each cube $ Q$:
	%%  ENUMERATE
\begin{enumerate}
\item  $ b_Q$ is supported on $ Q$ and $ \int _{Q} b_Q (x) \; dx = \lvert  Q\rvert $. 
\item  $ \lVert b_Q\rVert_{p} \le \mathbf A|Q|^{1/p}$. 
\end{enumerate}
%% ENUMERATE

\end{definition}
%%%%%%%%%%%%%%%%%%%%%%%%%%%%%%  DEFINITION DEFINITION DEFINITION

We aim to prove the following local $Tb$ theorem; denote $p'=p/(p-1)$.

%%%%%%%%%%%%%%%%%%%%%%%%%%%%%% THEOREM THEOREM THEOREM
\begin{theorem}\label{t.main1}  
Fix $ 1< p_1, p_2< \infty $ so that $1/p_1 + 1/p_2\le 1$.
Suppose $T$ is a Calder\'on--Zygmund operator for which
there	 are systems $ \{b ^{j} _{Q}\}$ of $p_j$-accretive functions, $j\in \{1,2\}$, with a constant $ \mathbf A$, 
	satisfying the following testing condition: there is a constant $ \mathbf T _{\textup{loc}}$ so that for all cubes $Q$, 
	\begin{equation*}
		\int _{ Q} \lvert  T b_Q^1\rvert ^{p_2'}  \le \mathbf T ^{p_2'} _{\textup{loc}} \lvert  Q\rvert \,, 
		\qquad 
		\int _{ Q} \lvert  T^{\ast} b_Q^2\rvert ^{p_1'}  \le \mathbf T ^{p_1'} _{\textup{loc}}\lvert  Q\rvert\,.
\end{equation*}
Then, we have a quantitative estimate $ \mathbf T \lesssim_{n,\eta,p_1,p_2, \mathbf A} 1+ \mathbf T _{\textup{loc}}$
for the operator norm of $T$.
\end{theorem}
%%%%%%%%%%%%%%%%%%%%%%%%%%%%%% THEOREM THEOREM THEOREM

In the case of perfect dyadic operators,
the full range $1<p_1,p_2< \infty$ of exponents  is allowed, as was shown in \cite[p. 48]{MR1934198}.
It was also hoped that the result could be lifted to the continuous 
case. This lifting turned out to be a difficult problem:
some of the direct methods \cite{0705.0840,1011.0642} to attack it
require assumptions that are  stronger than the duality assumption.
Theorem \ref{t.main1}  is due to Auscher--Yang \cite{MR2474120}, who provide an 
indirect argument---a
reduction to the perfect  dyadic  case.
The Auscher--Yang paper does not reach the  difficult case $1/p_1+1/p_2>1$, which
is also known as the `Hofmann's problem' as it was
 emphasized by Hofmann in \cite{MR2664559}. This problem was 
partially solved by
Auscher--Routin \cite{1011.1747} via adapting the
 Beylkin--Coifman--Rokhlin (BCR)  algorithm, see \cite{MR1085827,MR1110189}, as well as the martingale 
transform inequalities; at the same time, Auscher--Routin obtain a direct proof of Theorem \ref{t.main1}.
An essentially  full solution to Hofmann's problem 
has very recently been obtained by T. Hyt\"onen and 
F. Nazarov, \cite{hytonen_nazarov}. By applying perturbation techniques for both the operator and the accretive functions, they obtain 
a variant of Theorem
\ref{t.main1} for  $1<p_1,p_2<\infty$.

Our main contribution is an alternate direct proof 
of Theorem \ref{t.main1}.  
It is  desirable to have such proofs from the viewpoint of extensions of the argument to other settings.
%, such  as that of elliptic PDEs \cite{MR2664559}.
As an example,
in  the literature \cite{1011.0642, 1201.0648, MR1909219} on the local $Tb$ theorem in the non-homogeneous setting \cite{MR1909219}
one encounters stronger $L^\infty(\mathbf{R}^n)$ (or $\mathrm{BMO}$) conditions on 
$Tb_Q$'s, as well as on test functions $b_Q$.
%In the search after relaxation of these conditions
%one faces complications  that
%arise from the feature that the underlying measure
%need not be doubling.
 Some of the
techniques in the present paper have been subsequently applied
to relax these conditions in the case of square functions, \cite{lacey_martikainen}.
It even seems plausible that a variant of Theorem \ref{t.main1} could
be recovered in the non-homogeneous setting; see also \cite{lacey_vahakangas}.

\subsection*{Outline of the proof}
Let us turn to a discussion of the proof technique.  
As is quite common, absorbtion parameters enter into the proof at 
several stages, permitting us to  resort to  the assumed 
 finite---but non-quantitative---norm bound on $ T$, provided it is multiplied  by a small  absorption  parameter.  
We use the well-known non-homogeneous techniques of \cite{NTV1},  in particular, the 
powerful technique of `good cubes'.  
In the local $ Tb$ setting, there is however a delicate problem with the typical method of restricting to the  
good cubes,
as is pointed out by Hyt\"onen--Martikainen \cite{1011.0642}*{Remark 4.1}. 
An important innovation of the present paper is the  corona construction, which
enables us to restrict to good cubes in a natural way.
This construction depends on two random dyadic grids, $ \mathcal D ^{1}$ and $ \mathcal D ^{2}$,  
 that are defined on 
independent probability spaces $ \Omega ^{j}$, $ j=1,2$. 
A cube $ Q\in \mathcal D ^{1}$  is called \emph{bad}, if it is close to the boundary of some 
significantly larger cube in the  other grid,  $ \mathcal D ^{2}$.  
The badness of $Q$ is an event in $ \Omega ^{2}$ with probability that can be made arbitrarily small,  giving rise to an absorption parameter.   
A cube $Q$ is \emph{good}, if it is not bad.

Let us describe the corona construction in three steps.
First, by a  $T1$ theorem, \cite{MR763911},
it suffices consider the bilinear form $ \langle  T \widetilde  f_1 , \widetilde   f_2 \rangle$, where $ \lvert \widetilde  f_1\rvert = \lvert  \widetilde f_2\rvert = \mathbf 1_{Q ^{0}}  $ 
for a fixed cube $ Q ^{0}$.  
One projects $ \widetilde  f _1 $ onto the good cubes, calling the result $ f_1$, 
which can be viewed as a function of $ \Omega ^{1} $ and $ \Omega ^{2}$. This 
also contributes an error term, that is small in all $ L ^{p}$ spaces on average, and is treated by the first of several absorption arguments.
One then makes a standard selection of stopping cubes $ \widetilde{\mathcal S} ^{j} \subset \mathcal D ^{j}$ and local testing functions 
$ b ^{j} _{S}$ for $ S\in \mathcal S ^{j}$.  The stopping cubes $ \widetilde{\mathcal S} ^{j}$ is a sparse
collection, in particular, 
it is a Carleson sequence of cubes.  

In the next step, we construct functions $ \beta ^{1}_S$ by projecting $  b ^{1} _{S}$ 
\emph{away from} those bad cubes  which  themselves have $ S$ as a parent in $\widetilde{ \mathcal S} ^{1}$.
By doing so, we gain the following  desirable feature:  the twisted martingale difference 
of  $ f_1$, with respect to 
$ \beta ^{1}_S$ and over a bad cube $ Q$ with $\widetilde{\mathcal{S}}^1$ parent $S$, 
will  typically be zero.  
On the downside,  $ \beta ^{1} _{S}$ is now a function of $ \Omega ^{1}$ and $ \Omega ^{2}$, 
and the original collection of stopping cubes $\widetilde{ \mathcal S} ^{1}$ is 
not so well adapted to the $ \beta ^{1} _{S}$.   On the other hand, favorably to us,
$ \beta ^{1} _{S}$ can be viewed
as small perturbation of $b^1_S$. 

In the last step, to adopt the usage of perturbed functions $ \beta ^{1} _{S}$ in twisted martingale 
differences, 
one cannot run the stopping cube selection process again, due to the unacceptable dependices on $ \Omega ^{1}$ 
and $ \Omega ^{2}$.  Instead, one  invokes absorbtion, arguing that one can \emph{truncate} the stopping tree 
$\widetilde{ \mathcal S} ^{1}$ inside a set $ B ^{1}$  that is small on average.
The corona construction is now   described, and its details take up \S\ref{s.corona}, which is  almost  half 
the length of this paper. 

There are also  tools in \S\ref{s.inequalities} that are useful, namely martingale transform inequalities for twisted martingale differences, 
and the associated half-twisted inequalities that are \emph{universal}, in that they hold in all $ L ^{q}$-spaces. 
These inequalities also play a crucial role in \cite{1011.1747}*{Lemma 5.3} and in
 \cite{lv-perfect}.

Turning to the remaining part of the argument, one is in a familiar situation 
\cites{NTV1}
in the sense that 
only \emph{good} cubes 
$ P \in \mathcal D ^{1}$ and $ Q\in \mathcal D ^{2}$ need to be considered.  
The double sum over $ P, Q$ is reduced, by symmetry, to  the case of  $ \ell P \ge \ell Q$, and this sum is further decomposed 
into subcases according to the position  and size  of $ Q$ relative to $ P$. 
The case of $ Q$ deeply inside $ P$ admits 
a  direct control, by using the twisted martingale transform inequalities;
this `inside' case incorporates the  paraproduct term. 
For experts we remark that we do not appeal to Carleson measure arguments at any stage of the argument; in this we follow  \cites{1108.2319,1201.4319,1011.1747}.
The case of $ P$ and $ Q$ having the same approximate size 
 and position  requires  new perturbation inequalities for the 
twisted martingale transforms.  
This `diagonal' case is the hardest one in many 
existing arguments,
including ours.
A potentially troublesome case is when $ Q \subset 3P \setminus P$ and
 $Q$ is substantially smaller than $P$; however, due to goodness, $Q$ is still relatively far from the boundary of $ P$.
We address  this `nearby' case  by exploiting the smoothness condition on the kernel $ K$, and the universal half-twisted inequalities.  
 The remaining `far' case depends upon standard off-diagonal estimates for singular integrals,
and universal martingale transform inequalities.

\subsection*{Notation}
For a cube $ Q$, $ \langle f  \rangle_Q := \lvert  Q\rvert ^{-1} \int _{Q} f \; dx  $, 
and $ \ell Q = \lvert  Q\rvert ^{1/n} $ is the side length of the cube. $ A \lesssim B$ means that 
$ A \le C \cdot B$, where $ C$ is an unspecified constant which needs not be tracked.
The distances in $\mathbf{R}^n$ are measured in terms of the supremum norm, 
$\lvert x\rvert = \rVert x\rVert_\infty$ for $x\in\mathbf{R}^n$.
Given $ Q \in \mathcal D ^{j}$, we denote by $ \textup{ch}(Q)$ the $ 2 ^{n}$ dyadic children of $ Q$. 
Given $ \mathcal S\subset \mathcal D ^{j}$, we write  $ \textup{ch} _{\mathcal S} (S)$ for the $ \mathcal S$-children of $ S\in \mathcal S$: 
these are the maximal elements $ S'$ of $ \mathcal S$ that are strictly contained in $ S$.
For a  cube $ Q\in \mathcal D ^{j}$, that is contained in a cube in $\mathcal{S}$,
we take $ \pi _{\mathcal S} Q$ to be the $ \mathcal S$-parent of $ Q$: this is the minimal element of $ \mathcal S$ that 
contains $ Q$.

\subsection*{Acknowledgements} The authors would like to thank the referee for useful comments.
%Research supported in part by grant NSF-DMS 0968499, 
%and  a grant from the Simons Foundation (\#229596 to Michael Lacey).
%A.V.V.  was
%supported by the School of Mathematics, Georgia Institute of Technology, and 
%by the Finnish Academy of Science and Letters, Vilho, Yrj\"o and
%Kalle V\"ais\"al\"a Foundation.

%%%%%%%%%%%%%%%%%%%%%%%%%%%%%% SECTION  SECTION SECTION
%%%%%%%%%%%%%%%%%%%%%%%%%%%%%% SECTION  SECTION SECTION 
\section{The Corona} \label{s.corona}

%As is common \cite{MR1934198,1011.1747,MR2664559}, we will prove that for each cube $ Q^0$, 
%\begin{equation*}
%	\int _{Q^0} \lvert  T \mathbf 1_{Q^0}\rvert \; dx \lesssim  (1+\mathbf T _{\textup{loc}}) \lvert  Q^0\rvert \,, 
%\end{equation*}
%and, by symmetry,  the same inequality, with $ T$ replaced by $ T ^{\ast} $  holds.  
%Indeed, we will consider two functions $ \widetilde  f _1  , \widetilde f_2$ with $ \lvert \widetilde   f_1 \rvert=\lvert \widetilde  f_2\rvert =  \mathbf 1_{Q^0} $, and the bilinear 
%form $ \langle T \widetilde  f_1,\widetilde  f_2\rangle$.   
%(Here and below, tildes denote provisional objects.) 
%\color{blue} What about locally bounded kernel-assumption?\color{black}

It is a straightforward consequence of the $T1$ theorem, \cite{MR763911},
that 
\[
\mathbf{T} \lesssim 1 
+ \sup_{Q\subset\mathbf{R}^n\textup{ cube}} 
\lvert Q\rvert^{-1} \lVert \mathbf{1}_{Q} T^\ast \mathbf{1}_Q \lVert_{L^1}
+\sup_{Q\subset\mathbf{R}^n\textup{ cube}} 
\lvert Q\rvert^{-1}\lVert \mathbf{1}_{Q} T\mathbf{1}_Q \lVert_{L^1}\,.
\]
Without loss of generality, we can
assume that the last term dominates.
Fix a cube $ Q ^{0} $ for which
\begin{equation}\label{e.witness}
\mathbf{T}  \lvert  Q ^{0}\rvert  
\lesssim 
\lVert \mathbf 1_{Q^0}  T \mathbf 1_{Q ^{0}}\rVert_{L^1}\,.
\end{equation}
For notational convenience, let us take
two functions $ \widetilde  f _1  , \widetilde f_2$ such that $ \lvert \widetilde   f_1 \rvert=\lvert \widetilde  f_2\rvert =  \mathbf 1_{Q^0} $ and
$\lVert \mathbf 1_{Q^0}  T \mathbf 1_{Q ^{0}}\rVert_{L^1} = \langle T \,\widetilde  f_1,
\widetilde  f_2\rangle$.
The main purpose of the present section is to devise a corona-type decomposition, which helps us to 
restrict to good cubes, after which it will be straightforward to complete the proof of
the following lemma. 

%%%%%%%%%%%%%%%%%%%%%%%%%%%%%% LEMMA LEMMA LEMMA
\begin{lemma}\label{l.sharp} Fix $ 0< \upsilon _0 < 1$. There are
	functions $ f_1 $ and $ f_2 $, and
a constant  $ C>0 $ independent of both 
$\mathbf{T}$ and $\mathbf{T}_{\textup{loc}}$, such that the following inequalities hold:
	\begin{gather}\label{e.sharp<eta} 
		\lVert\widetilde  f_j- f_j  \rVert_{2}   < \upsilon_0   \lvert  Q ^{0}\rvert ^{1/2}  \,, \qquad j=1,2\,, 
		\\  \label{e.sharpIP}
		\bigl\lvert \langle T f_1,  f_2\rangle\bigr\rvert < \{C   (1+\mathbf T _{\textup{loc}})  + \upsilon_0 \mathbf T \}\lvert  Q ^{0}\rvert   \,. 
\end{gather}
\end{lemma}
%%%%%%%%%%%%%%%%%%%%%%%%%%%%%% LEMMA LEMMA LEMMA

This lemma and an absorption argument complete the proof of Theorem \ref{t.main1}. 
The construction of the corona
is rather complicated. 
It will be highly dependent upon certain random constructions, and there will be several absorption parameters that lead to the constant
$ \upsilon_0 $.  
The main advantage of our corona construction is that it allows
us to restrict to the good cubes in a natural manner; this and other useful features admit a straightforward proof of
 inequality \eqref{e.sharpIP}.

%%%%%%%%%%%%%%%%%%%%%%%%%%%%%% SECTION  SECTION SECTION
%%%%%%%%%%%%%%%%%%%%%%%%%%%%%% SECTION  SECTION SECTION 
\subsection{Random Grids} %\label{s.}
We make use of so-called random grids,  due to Nazarov--Treil--Volberg \cite{MR1909219}.
These turned out to be of fundamental importance,
see \cite{MR1756958,1201.4319,0911.4387,MR2912709} for examples. 

We will have a random grid $\mathcal D ^{1}$ for the functions
 $\widetilde f_1,f_1$ and a 
random grid $ \mathcal D ^{2} $ 
for the functions $\widetilde f_2,f_2$. 
These random grids are constructed as follows. 
Let $ \mathcal D^0$ be the standard dyadic grid in $ \mathbf R ^{n}$.  
For  a fixed cube $\widehat{Q}\in\mathcal{D}^0$, let us consider 
the translated cube
\begin{equation*}
	Q:=\widehat{Q} \dot+ \omega ^{1} := \widehat{Q}+\sum_{j\,:\,2^{-j}< \ell Q } 2^{-j}\omega_j ^{1}\,,
%	Q + \Bigl( \sum_{ j_1 \;:\; 2 ^{j_1 } {<} \ell (Q) } \omega _{1,j_1} 2 ^{j_1}  ,\dotsc, 
%	\sum_{ j_1 \;:\; 2 ^{j_n} {<} \ell (Q) } \omega _{1,j_n} 2 ^{j_n} \Bigr) \,,
\end{equation*}
which is a function of $ \omega ^{1} \in \Omega ^{1}  :=(\{0,1\}^n )^{\mathbf{Z} }$.
Denote  $\mathcal{D}^1=\{\widehat{Q} \dot+ \omega ^{1}\,:\,\widehat{Q} \in\mathcal{D}^0\}$.
The natural uniform probability measure  $\mathbb{P}^1$  is placed upon $ \Omega ^{1} $. That is,
each component   $\omega_j^1$, $j\in\mathbb{Z}$, has an equal probability $2^{-n}$ of taking any of the $2^n$ values, and all the components are independent of each other. The expectation
 with respect to $\mathbb{P}^1$ is denoted by $\mathbb{E}^1$. 
Define $ \Omega ^{2}$ in the same manner, with an independent copy of $ \Omega ^{1} $.  
It will be  important to distinguish between these two copies, so we  write $ \omega ^{j} \in \Omega ^{j}$ 
for the elements of the probability space that define $ \mathcal D ^{j}$.
 The product  $\mathbb{P}^1\otimes \mathbb{P}^2$
is denoted by $\mathbb{P}$, and the corresponding expectation 
$\mathbb{E}^1\mathbb{E}^2$ is denoted by $\mathbb{E}$.

%As our construction has several parts,  we will often simply write $ Q$ for a cube in $\mathcal{D}^j$,
%instead of the heavier notation $ \widehat{Q} \dot+ \omega ^{j} $ with 
%$\widehat{Q}\in\mathcal{D}^0$.  
%Those parts of the argument that 
%depend upon the construction of the random grids will be clearly identified. 

We need notation.
Define the familiar \cite{MR1909219,1011.0642,MR2912709} and convenient number
\begin{equation} \label{e.epsilon}
 \epsilon :=  \frac{\eta}{2(\eta + n)}\,.
\end{equation} 
Throughout  $ r\ge 3/ \epsilon$ should be thought of  as a large integer, 
which satisfies condition (3) below,
 and whose exact value is assigned later.  
We say that a cube $ Q\in\mathcal{D}^1$ is 
\emph{bad}, if there is $P\in\mathcal{D}^2$ such that
$\ell(P)\ge 2^r \ell(Q)$ and 
$ 	 \textup{dist} (Q, \partial P) \le (\ell Q) ^{\epsilon } (\ell P) ^{1- \epsilon}$.
Otherwise, $Q$ is \emph{good}.
The definitions for  $Q\in \mathcal D ^{2}$ are similar.
The following properties are well-known for a cube $Q\in\mathcal{D}^1$:
%%  ENUMERATE
\begin{enumerate}
\item  The goodness/badness
of  $ Q $ 
is a random variable on $ \Omega ^{2}$;
%\item The position and
%goodness/badness of $Q$ are independent random variables---indeed, the former is a function
%of $\omega^1\in \Omega^1$;
\item The probability
$ \pi _{\textup{good}} := \mathbb P^2  (\textup{$Q$ is good})$ is independent of  $ Q$;
\item  $  \pi _{\textup{bad}}:=1- \pi _{\textup{good}}  \lesssim 2 ^{- \epsilon r}$, provided $ \epsilon r$ is sufficiently large.
%	The integer $ r$ will need to be so large that $  \pi _{\textup{bad}}  $ is sufficiently small, as a function of $ \mathbf A$, 
%	$ p_1$, $ p_2$, and dimensional constants. 
\end{enumerate}
%% ENUMERATE
 Define the good and bad projections by $ I= P ^{j} _{\textup{good}}+ P ^{j} _{\textup{bad}}$, where 
 \begin{equation*}
 	 P ^{j} _{\textup{good}} \phi := \sum_{Q \in \mathcal D ^{j}\;:\;  \textup{$ Q$ is good}} D _{Q} \phi \,, \qquad j=1,2\,.
\end{equation*}
Here $ D _Q \phi=\sum_{Q'\in\textup{ch}(Q)}
\{ \langle \phi\rangle_{Q'} - \langle \phi\rangle_Q\} \mathbf{1}_{Q'}$ is the usual martingale difference
associated with $Q$. 

We have the following proposition on the  bad projections;
The constant $ 0<c_q<1$ that appears in the exponent on the right will 
be a function of $ p_1$ and $ p_2$. In the sequel, we suppress this dependence in notation, writing 
only $ 2 ^{- c \epsilon r}$.

%%%%%%%%%%%%%%%%%%%%%%%%%%%%%% PROPOSITION PROPOSITION PROPOSITION
\begin{proposition}\label{p.bad} If $ 1< q < \infty $  and $\{j,k\}=\{1,2\}$,  then there is a constant $ c_q >0 $ so that 
	\begin{equation}\label{e.bad}
		 \mathbb{E}^k \lVert  P ^{j} _{\textup{bad}} \phi \rVert_{q}^q \lesssim 2 ^{- c_q \epsilon r} \lVert \phi \rVert_{q}^q\,.
\end{equation}
Here $\omega^j\in\Omega^j$ is fixed, and $\phi\in L^q$ is any function
that is independent of sequences $\omega^k\in \Omega^k$.
\end{proposition}

%%%%%%%%%%%%%%%%%%%%%%%%%%%%%% PROPOSITION PROPOSITION PROPOSITION

%%%%%%%%%%%%%%%%%%%%%%%%%%%%%% PROOF PROOF PROOF
\begin{proof} 
The basic idea is to apply the
Marcinkiewicz interpolation theorem to the linear
	operator $P ^{j} _{\textup{bad}}:L^q(dx)\to L^q(\mathbb{P}^k\otimes dx)$.
	The projection to bad cubes is a martingale transform 
	\cite{MR1108183}, hence the following inequality with no decay holds, 
	\begin{equation*}
		\mathbb{E}^k \lVert  P ^{j} _{\textup{bad}} \phi \rVert_{p} ^{p} 
		\le \sup\, \{ \lVert  P ^{j} _{\textup{bad}} \phi \rVert_{p} ^{p}\,:\omega^k\in\Omega^k \}
		\lesssim  \lVert \phi \rVert_{p} ^{p} \,,\quad 1<p<\infty\,.
\end{equation*}
	Thus,
	it suffices to verify the claimed decay for $ q=2$. To this end, by independence,
	\begin{align*}
		\mathbb{E}^k \lVert  P ^{j} _{\textup{bad}} \phi \rVert_2 ^2 
				= \mathbb{E}^k \sum_{ \substack{Q \in \mathcal D ^{j} \\ \textup{$ Q$ is bad}}}  \lVert D_Q \phi \rVert_{2} ^2 
	 =  \pi _{\textup{bad}} 
		\sum_{Q \in \mathcal D ^{j}}  \lVert D_Q \phi \rVert_{2} ^2 
		=\pi _{\textup{bad}}		   \lVert \phi \rVert_{2} ^2\,.
\end{align*}
 Indeed, both $\mathcal{D}^j$ and $\lVert D_Q \phi \rVert_{2} ^2$ for $Q\in\mathcal{D}^j$
are independent of $\omega^k\in \Omega^k$, and the badness of $Q\in\mathcal{D}^j$ is 
a random variable on $\Omega^k$, $\{j,k\}=\{1,2\}$.
\end{proof}
%%%%%%%%%%%%%%%%%%%%%%%%%%%%%% PROOF PROOF PROOF

%%%%%%%%%%%%%%%%%%%%%%%%%%%%%% SUBSECTION SUBSECTION SUBSECTION SUBSECTION
 %%%%%%%%%%%%%%%%%%%%%%%%%%%%%% SUBSECTION SUBSECTION SUBSECTION SUBSECTION 
 \subsection{Selection of $ f_j$} 
We will prove Lemma~\ref{l.sharp} by averaging over random grids. 
Fix $ j\in \{1,2\}$.
Let $ \mathcal A ^{j} _{\ast }$ denote all (at most $ 2 ^{n}$) cubes $ Q\in \mathcal D ^{j}$ such that 
 $ Q\cap Q ^{0} \neq \emptyset $ and $ \ell Q ^{0} \le \ell Q < 2 \ell Q ^{0}$.  
 Let $ \mathcal A ^{j}$ be all cubes  in $\mathcal{D}^j$  that are contained in some $ Q\in \mathcal A ^{j} _{\ast }$.  
Recall that the function $ \widetilde f _j $ is
chosen in connection with \eqref{e.witness}, and it
is equal to $ \mathbf 1_{Q ^{0}}$ in absolute value.  
 We define an approximate $f_j$ of this function to be
  \begin{equation*}
  	  f _{j} := \sum_{Q\in \mathcal A ^{j} _{\ast }} \langle \widetilde f _{j} \rangle_Q \mathbf 1_{Q} 
  	  + \sum_{\substack{Q \in \mathcal A ^{j}\\ \textup{$ Q$ is good}}} D_Q \widetilde f_j \,. 
\end{equation*}
In the view of Proposition~\ref{p.bad}, we have
\begin{equation}\label{e.f_diff_small}
\mathbb{E} \lVert \widetilde{f}_j - f_j\rVert_2^2 \lesssim
2^{-c\epsilon r} \lvert Q^0\rvert\,.
\end{equation}
Hence, it suffices to estimate 
 $ \mathbb E  \lvert\langle T f_1, f_2 \rangle\rvert $.

The functions $ f_j$  lie in $ BMO$:---a dyadic variant associated with the  grid $\mathcal{D}^j$.  It follows from 
 the associated John--Nirenberg inequality that 
 \begin{equation}\label{e.f_bmo}
 	 \lVert f_j\rVert_{q}  \lesssim \lvert  Q ^{0}\rvert ^{1/q}\,, \qquad 1< q < \infty \,,
\end{equation}
 with the implied constant independent of  sequences $\omega^1$ and $\omega^2$. 
The fact that the functions $f_j$ can nevertheless be unbounded creates a minor set of difficulties for us.  
 
%%%%%%%%%%%%%%%%%%%%%%%%%%%%%% SUBSECTION SUBSECTION SUBSECTION SUBSECTION
 %%%%%%%%%%%%%%%%%%%%%%%%%%%%%% SUBSECTION SUBSECTION SUBSECTION SUBSECTION 
 \subsection{The Setup for Stopping Cubes Construction}%\label{ss.} 

 In order to accommodate the reduction to good cubes, we will need a significant modification of the 
 usual selection process of stopping trees and local $ b$ functions.  The following definition
 will help 
 explain the end result that we are after; it is convenient
 to denote $T^1=T$ and $T^2=T^\ast$.

%%%%%%%%%%%%%%%%%%%%%%%%%%%%%%  DEFINITION DEFINITION DEFINITION
\begin{definition}\label{d.stoppingData}  
	Fix constants $ 0< \tau,\delta < 1$, and let $\{j,k\}=\{1,2\}$.
	A collection of integrable functions
	$   \{\beta ^{j} _S \;:\; S\in \mathcal S ^{j}\subset \mathcal D^j\} $
	 is a \emph{stopping data}  (a \emph{perturbed stopping data})
	 for 
	  a collection $ \mathcal G ^{j} \subset \mathcal D ^{j}$ of cubes if the following 
	conditions hold with $\mathbf{A}_j=1/2$, $\mathbf{B}_j=\delta ^{-1}  \mathbf A^{p_j}$, 
	and $\mathbf{C}_j=\delta ^{-1} \mathbf T _{\textup{loc} } ^{p_k'}$
	(in the case of perturbed stopping data: $\mathbf{A}_j=1/4$, $\mathbf{B}_j\lesssim \delta ^{-1}  \mathbf A^{p_j}$, 
	and $\mathbf{C}_j\lesssim \delta ^{-1} \mathbf T _{\textup{loc} } ^{p_k'}+\upsilon_1^{p_k'}\mathbf{T}^{p_k'}$ for some constant $0<\upsilon_1<1$):
	%%  ENUMERATE
\begin{enumerate}
\item Every $ Q\in \mathcal G ^{j}$ is contained in some $ S\in \mathcal S ^{j}$. The same 
holds for every 
	child  $ Q'\in \textup{ch} (Q)$, whose parent $ \pi _{\mathcal S ^{j}} Q' $ need not equal $ \pi _{\mathcal S ^{j}} Q$, 
	even if $ Q$ is a minimal cube in $ \mathcal G ^{j}$.  
\item If $ Q\in \mathcal G ^{j}$  with  $ \pi _{\mathcal{S}^j} Q=S$
 (or $Q\in \textup{ch}(R)$ with $R\in\mathcal{G}^j$ and
$\pi_{\mathcal{S}^j}Q=S$),  then (a)---(c);
	%%  ENUMERATE
\begin{enumerate}
\item  $\langle  \beta ^j _{S} \rangle_Q \ge \mathbf{A}_j   $\,;  \qquad (Don't divide by zero)
	
\item $  \langle    \lvert M \beta ^j _{S}\rvert ^{p_j}  \rangle_Q \le \mathbf{B}_j $\,;  \qquad (Local norm of $ M\beta ^j_S$ controlled)
	
\item $  \langle  \lvert T^j \beta ^j _{S}\rvert ^{  p_k'}  \rangle_Q \le \mathbf{C}_j  $\,;
	\qquad (Local norm of $ T^j \beta ^j _{S}$ is controlled)
\end{enumerate}
%% ENUMERATE
%Note that (b), (c) hold with $Q$ replaced by $ Q'\in \textup{ch} (Q)$, as Lebesgue measure is doubling.  
\item $ \sum_{ S' \in  \textup{ch}_{\mathcal{S}^j}(S)  } \lvert  S'\rvert \le  \tau  \lvert  S\rvert $ for all $ S\in \mathcal S^j$, i.e., $\mathcal{S}^j$ is a sparse collection of cubes.
\end{enumerate}
%% ENUMERATE
%(Below, the constant $ 0< \delta < 1$ is chosen so small that in the last property $ 0 < \tau = \tau _{\mathbf A} < 1$.) 
For $ Q\in \mathcal G ^{j}$ and $\phi\in L^1_{\textup{loc}}$, we define a \emph{twisted martingale difference} by
\begin{equation} \label{e.twist}
	\Delta_Q  ^{\beta^j} \phi := 
	 \sum_{Q'\in \textup{ch} (Q)} 
	\biggl\{
	\frac {\langle \phi \rangle _{Q'}} {\langle \beta^j_{\pi_{\mathcal{S}^j} Q'} \rangle _{Q'} } \beta_{ \pi _{\mathcal S^j}Q'} ^{j}
	-
	\frac {\langle \phi \rangle _{Q}} {\langle \beta^j_{\pi _{\mathcal S^j} Q} \rangle_Q}  \beta_{\pi_{\mathcal{S}^j} Q} ^{j}
	\biggr\}  \mathbf 1_{Q'} \,.
\end{equation}
This is well defined, as $Q$ has an $ \mathcal S ^{j}$ parent,
 and there is no division by zero; see conditions (1) and (2a). 
We also define a \emph{half-twisted martingale difference} by 
\begin{equation}\label{e.halfDef}
	\widetilde D_Q  ^{\beta^j} \phi := 
	\biggl\{\sum_{\substack{Q'\in \textup{ch} (Q)\\ \pi _{\mathcal S ^{j}}Q= \pi _{\mathcal S ^{j}} Q' }} 
	\frac {\langle \phi \rangle _{Q'}} {\langle \beta^j_{\pi_{\mathcal{S}^j} Q'} \rangle _{Q'} }  \mathbf 1_{Q'}
	\biggr\} 
	-
	\frac {\langle \phi \rangle _{Q}} {\langle \beta^j_{\pi _{\mathcal S^j} Q} \rangle_Q} 
	\mathbf 1_{Q} \,. 
\end{equation}
Observe that here we do not multiply by a $ \beta^j$ function, and the 
sum 
over the children 
excludes those with a different $ \mathcal S^j$ parent (in particular, there is no change
in the $ \beta^j$ function: $\pi_{\mathcal{S}^j} Q' =\pi_{\mathcal{S}^j} Q$).  
\end{definition}
%%%%%%%%%%%%%%%%%%%%%%%%%%%%%%  DEFINITION DEFINITION DEFINITION

The following Lemma provides the reduction to good cubes. In particular, it helps us to
eliminate the martingale  differences that are associated with bad cubes, 

%%%%%%%%%%%%%%%%%%%%%%%%%%%%%% LEMMA LEMMA LEMMA
\begin{lemma}\label{l.GB} Suppose $ \Lambda > 1$
and $0<\upsilon_1<4^{-1-n}$.
Fix $j\in \{1,2\}$.
There is a collection $ \mathcal G ^{j} \subset \mathcal D ^{j}$  of cubes, 
	and  a perturbed  stopping data $\{ \beta ^{j}_S \;:\; S\in \mathcal S ^{j} \}$ for $ \mathcal G ^{j}$, so that conditions (1)---(4) hold:
	%%  ENUMERATE
\begin{enumerate}
\item  Every cube  $ Q\in \mathcal G ^{j}$ is good;
\item  For all $ Q \in \mathcal G ^{j}$, we have $\langle |f_j|\rangle_Q \le \Lambda$;
\item Suppose $ Q \in \mathcal G ^{j}$ with a child $ Q'$, and $ S\in \mathcal S  ^{j}$ with  $ \pi _{\mathcal S ^{j}} Q \subset S$. Define a constant  $ \lambda _{Q'} $ by
	\begin{equation} \label{e.BNDD}
		\lambda_{Q'}\mathbf{1}_{Q'} :=\mathbf 1_{Q'} \sum_{\substack{ P \in \mathcal G ^{j} \;:\; P\supset Q\\  \pi _{\mathcal S ^{j}}P= S} } \widetilde D _P^{\beta^j} f_j \,. 
\end{equation}
Then,  we have $ \lvert  \lambda_{Q'} \rvert \lesssim \Lambda  $.  
\item   Assuming $\Lambda^{-1}+\Lambda\upsilon_1^{-1}\cdot 2^{-c\epsilon r}<1$, 
 there holds  
		\begin{equation} \label{e.goodSum}
		\begin{split}
			&\mathbb E 	\Bigl\lvert \langle T f_1, f_2\rangle - \sum_{P\in \mathcal G ^{1} }	\sum_{Q\in \mathcal G ^{2} } 
			\langle  T  \Delta _P ^{\beta^1 } f_1,  \Delta _Q ^{\beta^2 }f_2\rangle \Bigr\rvert 
			\\&\qquad\qquad\le C_1\bigl\{
			1+ \mathbf T _{\textup{loc}}    
			+ ( \upsilon _1+ \Lambda ^{-1}  + 
			\Lambda \upsilon _1^{-1}\cdot   2 ^{- c\epsilon r}) \mathbf T  \bigr\}  \lvert Q ^{0}\rvert .
			 \end{split}
\end{equation}
Here %$ C_0 = C_0 (p_1, p_2, n, \mathbf A, r, \Lambda , \upsilon _1)$, 
%while 
$ C_1 = C_1 (p_1,p_2,n, \mathbf A)$ does not depend upon 
the absorption parameters $\upsilon_1$, $\Lambda$, $ r$.  
\end{enumerate}
%% ENUMERATE
\end{lemma}

Before the lengthy  proof of this lemma, let us indicate its usage.
%%%%%%%%%%%%%%%%%%%%%%%%%%%%%% LEMMA LEMMA LEMMA

%It is clear that $0<\upsilon_1\ll 1$, and $ \Lambda\gg 1$, and $r\gg 1$ are absorbtion parameters.

\subsection*{A conditional proof of Lemma~\ref{l.sharp}}
In order to complete the proof of Lemma~\ref{l.sharp}, 
it remains to verify Lemma \ref{l.GB} and the following inequality, 
\begin{equation}\label{e.GS<}
	\begin{split}
	\Bigl\lvert 
 \sum_{P\in \mathcal G ^{1} }	\sum_{Q\in \mathcal G ^{2} } 
			\langle  T  \Delta _P ^{\beta^1 } f_1,  \Delta _Q ^{\beta^2 }f_2\rangle 
			\Bigr\rvert &\le 
			 \bigl\{C_2
			\{ 1+ \mathbf T _{\textup{loc}}\}
			 + C_3   r\upsilon _1 \Lambda^2 \mathbf T  \bigr\} 
			 \lvert  Q ^{0}\rvert\,.
	\end{split}
\end{equation}
 We emphasize that  inequality \eqref{e.GS<}
is uniform in $\omega^1$ and $\omega^2$, and that it is
distinct from  \eqref{e.goodSum}. The 
 constant $ C_3=C_3 (p_1,p_2,n, \eta, \mathbf A)$, that is
 independent of absorption parameters,
 and the product $ r\upsilon_1\Lambda^2 $ of absorption parameters appear on the right.  
The constant
 \[C_2 = C_2 (p_1, p_2, n, \eta,\mathbf A, r, \Lambda , \upsilon _1)\]
is allowed to depend also upon the absorption parameters.
Returning 
to the proof of Lemma \ref{l.sharp}, let us
consider 
 inequalities \eqref{e.f_diff_small}, \eqref{e.goodSum}, and
 \eqref{e.GS<}.
 By taking $ \Lambda >1$ sufficiently large, and then 
 choosing $r$ large enough and assigning $\upsilon _1=r^{-2}$, the proof is complete---apart from Lemma \ref{l.GB} and inequality \eqref{e.GS<}.  \qed

\medskip

At this stage, let us make several clarifying remarks.

%%%%%%%%%%%%%%%%%%%%%%%%%%%%%% REMARK REMARK REMARK
\begin{remark}\label{r.t+h}  Hyt\"onen and Martikainen \cite{1011.0642}*{Remark 4.1} have pointed to 
	serious concerns with some existing approaches to the reduction to good cubes in local $ Tb$ theorems. 
	The substance of the problem arises from the fact that the twisted martingale differences 
	depend upon the choice of grid, and the collection of local $ b$ functions, 
	making averaging arguments---such as the one used in the proof of Proposition~\ref{p.bad}---not transparently true. Our corona construction establishes a transparent reduction to 
	good cubes in \eqref{e.GS<}, and this
	is one of our main contributions.
	%This is a serious point, that affects certain existing arguments.  
\end{remark}
%%%%%%%%%%%%%%%%%%%%%%%%%%%%%% REMARK REMARK REMARK
 
 \begin{remark}
 The proof of inequality \eqref{e.GS<}, taken up \S\ref{s.innerProduct}--\S\ref{s.remaining}, is now largely standard in nature, following the lines of \cite{NTV1,V} and including 
	innovations from \cite{1108.2319,1201.4319} to avoid auxiliary Carleson measure estimates.  However, certain perturbation
	inequalities are needed when treating cubes that are nearby, both in size
	and position. There are also advantages for us:
%%  ENUMERATE
\begin{enumerate} 
\item We need only consider good cubes, which is the primary goal of the corona construction.  
	
\item
 By normalizing both $f_1$ and $f_2$ 
with a factor $\Lambda^{-1}$, 
the sums \eqref{e.BNDD} are bounded by $c\lesssim 1$, 
	which is related to the telescoping property needed in the control of paraproduct terms. 
This normalization is assumed
 in the beginning of \S\ref{s.inequalities}, and thereafter.
\end{enumerate}
\end{remark}
%% ENUMERATE

%%%%%%%%%%%%%%%%%%%%%%%%%%%%%% REMARK REMARK REMARK
\begin{remark}\label{r.notracking} 

	The dependence of the quantitative estimates on the parameters 
	aside from $\mathbf{T}$ and $\mathbf{T}_\textup{loc}$
%	$ \mathbf A, p_j$, and dimension, 
is 
not straight forward, and typically we do not  track it.  
However, we need to
track the dependence of a constant
$c$ on absorption parameters 
$r$, $\Lambda$, $\upsilon_1$,
if it appears
in an expression $c\cdot \mathbf{T}$.
\end{remark}
%%%%%%%%%%%%%%%%%%%%%%%%%%%%%% REMARK REMARK REMARK

The rest of this section is taken up with the proof of Lemma \ref{l.GB}.  

\subsection{Auxiliary stopping data}
Fix $j\in \{1,2\}$.
We construct  auxiliary
\emph{stopping data} $ \{b ^{j} _{S} \;:\; S\in \widetilde {\mathcal S } ^{j}\}$ 
 for the collection $ \mathcal A ^{j}$, which was defined when selecting the function $ f _{j}$. 
 The  \emph{perturbed stopping
 data} in  will be later constructed by using this auxiliary stopping data.
 The following construction of $ \widetilde{ \mathcal S} ^{j}$ and $ \{b _{S}^j \;:\; S\in \widetilde{ \mathcal S} ^{j}\}$ is 
fairly standard, and it only depends upon $ \omega^{j}$. 

 Initialize $ \widetilde{ \mathcal S} ^{j}$ to be $\mathcal A ^{j} _{\ast }$. For each
 cube $S$ in this collection,  consider the function $ b ^{j}_S$ given to us by the local $ Tb$ hypothesis, see the formulation of Theorem \ref{t.main1}.
 Add to $ \widetilde{ \mathcal S} ^{j}$ the maximal dyadic descendants  $Q\subset  S$
 which either fail  %the condition  $ \langle  b ^{j}_S\rangle _{S'} > 2^{-1}  $, or 
 any of the criteria  (a)---(c)  in Definition~\ref{d.stoppingData},
 with $\beta_S^j:=b_S^j$,
 or fail the condition
 \begin{equation}\label{e.inf}
 \inf_{x\in Q} M|b^j_S|^{p_j}(x) \le \delta^{-1} \mathbf{A}^{p_j}\,.
 \end{equation}
% (To accomodate an absorbtion parameter below, 
% we have made the first criteria stronger than in Definition~\ref{d.stoppingData}.) 
 Concerning these stopping conditions, 
 let  $E _S$  be the union of the maximal descendents $Q$ of 
 $ S$ such that  $ \langle  b ^{j}_S\rangle _{Q} < \tfrac 12$.
 We have, using the higher integrability of $ b ^{j}_S$, 
\begin{align*}
	\lvert  S\rvert = \int _{S} b ^{j}_S \;dx &= \int _{E_S} b ^j_S \; dx + \int _{S \setminus E_S} b ^{j}_S \; dx 
 \le \tfrac 12 \lvert S\rvert + \mathbf A  \lvert  S \setminus E_S\rvert ^{1/p_j'} \lvert  S\rvert ^{1/p_j}   \,. 
\end{align*} 
Hence, $ (2\mathbf A) ^{-p_j'} \lvert  S\rvert \le \lvert  S \setminus E_S\rvert  $.
Next, let us consider the union $F_S$ of the maximal descendants $Q$ of $S$, 
failing \eqref{e.inf} or  one of the mentioned criteria (b), (c). By inspection, we have
$\lvert F_S\rvert \lesssim \delta \lvert S\rvert$.
Therefore, with choice of $ \delta = \delta(p_j,n,{\mathbf A})$, we can 
continue the construction of $ \widetilde{\mathcal{S}}^{j}$ inductively
to meet conditions (1)---(2) and
the sparsness condition (3) in Definition~\ref{d.stoppingData} with $\tau=\tau(p_j,n,{\mathbf A})$.

 Below, we will refer to $ \widetilde {\mathcal S} ^{j}$, and its subsets, as collections of \emph{stopping cubes}.

 \subsection{Perturbation of the $ b$ functions}%\label{ss.}
 
In a departure from standard arguments, we modify the functions $ b ^{j} _{S}$,  $S\in\widetilde{\mathcal S}^j$, that are  already selected.
For $S\in\widetilde{\mathcal S}^j$, we define 
\begin{align}\label{e.beta}
	\beta ^j _{S}  := b ^j _S -  \widetilde \beta ^j_S \,, \quad \textup{where} \quad 
	\widetilde \beta ^j_S  := \sum_{\substack{Q\in\mathcal{A}^j \;:\; \pi _{\widetilde{\mathcal S} ^{j}}Q=S\\ \textup{$ Q$ is bad}}}  D _Q b ^j_S\,.
	\end{align}
We notice that the sum defining $  \widetilde \beta ^j_S $ is formed by using
the classical martingale differences that are associated with bad cubes  in $\mathcal{A}^j$
which have the same stopping parent.  
A particular care must be taken with these perturbations $ \beta ^{j} _S$, as they are now functions of both $ \omega ^{1}$ and $ \omega ^{2}$. 

Nevertheless,  $ \widetilde \beta ^{j}_S$ is a small function on average.

%%%%%%%%%%%%%%%%%%%%%%%%%%%%%% LEMMA LEMMA LEMMA
\begin{lemma}\label{l.tildeB} For $\{j,k\}=\{1,2\}$ and all $S\in\widetilde{\mathcal{S}}^j$,
  there holds 
	\begin{align}\label{e.tildeB-BMO}
		\lVert\widetilde \beta ^{j}_S \rVert_{BMO} & \lesssim 1\,,
		\\
 	\mathbb{E}^k \lVert\widetilde \beta ^{j}_S \rVert_{q}^q & \lesssim 2 ^{- c \epsilon r} \lvert  S\rvert \,, \qquad 1< q < \infty \,.  
\end{align}
\end{lemma}
%%%%%%%%%%%%%%%%%%%%%%%%%%%%%% LEMMA LEMMA LEMMA

%%%%%%%%%%%%%%%%%%%%%%%%%%%%%% PROOF PROOF PROOF
\begin{proof}
	Let $ Q\in\mathcal{D}^j$ be such that $ \pi _{\widetilde {\mathcal S} ^j}Q =  S$.
	Writing $\epsilon_{Q'}:=\mathbf{1}_{Q'\subset Q}
	\mathbf{1}_{\pi _{\widetilde{\mathcal S} ^{j}}Q'=S}
	\mathbf{1}_{\textup{$ Q'$ is bad}}$, we obtain
	\begin{align*}
	&\bigg(\int_Q \lvert \widetilde \beta ^j_S - \langle \widetilde \beta ^j_S\rangle_Q\rvert^{p_j} \; dx \bigg)^{1/p_j}
	=
\Bigl\lVert \sum_{Q'\subset Q } D_{Q'}  \widetilde \beta ^j_S \Bigr\rVert_{p_j}
\\
&= \Bigl\lVert 
	\sum_{Q' \in\mathcal{D}^j} \epsilon_{Q'} D_{Q'} (\mathbf 1_Q b ^j_S)
\Bigr\rVert_{p_j} 
\lesssim 
\lvert \lvert \mathbf{1}_Q b_S^j\rVert_{p_j}\le
\langle   \lvert  Mb ^j_S\rvert ^{p_1}  \rangle_Q ^{1/p_j} 
\lvert  Q\rvert ^{1/p_j}\lesssim \lvert Q\rvert^{1/p_j}\,.
\end{align*}
Here, we have appealed to the boundedness of martingale transforms, and the stopping rules.  
 The remaining cases either reduce to this, or are trivial. 
Hence the $ BMO$ assertion is true.

 Concerning the $ L ^{q}$ estimate, 
 we apply Proposition~\ref{p.bad} and
the John-Nirenberg inequality,
 \begin{align*}
 \mathbb{E}^k \lVert\widetilde \beta ^{j}_S \rVert_{q}^q 
 &=\mathbb{E}^k \biggl\lVert  P^j_{\textup{bad}}\biggl[  \sum_{\substack{Q\;:\; \pi _{\widetilde{\mathcal S} ^{j}}Q=S}}  D _Q b ^j_S
 \biggr]
\biggr \rVert_{q}^q 
\\& \lesssim
 2^{-c\epsilon r}\Bigl\lVert \sum_{\substack{Q \;:\; \pi _{\widetilde{\mathcal S} ^{j}}Q=S}} 
 D _Q b ^j_S \Bigr\rVert_{q}^q\lesssim 2^{-c\epsilon r}|S|\cdot
 \Bigl\lVert \sum_{\substack{Q \;:\; \pi _{\widetilde{\mathcal S} ^{j}}Q=S}}  D _Q b ^j_S \Bigr\rVert_{\mathrm{BMO}}^q.
 \end{align*} 
By arguing as above, we finish the proof.
\end{proof}
%%%%%%%%%%%%%%%%%%%%%%%%%%%%%% PROOF PROOF PROOF
%%%%%%%%%%%%%%%%%%%%%%%%%%%%%% SUBSECTION SUBSECTION SUBSECTION SUBSECTION
 %%%%%%%%%%%%%%%%%%%%%%%%%%%%%% SUBSECTION SUBSECTION SUBSECTION SUBSECTION 
 \subsection{Truncation of the Stopping Tree}%\label{ss.}

We will use the functions $ \beta ^{j} _{S}$ as the basis of  perturbed  
stopping data, see Lemma \ref{l.GB}, but 
the path to this is not yet clear for these reasons: (A) the functions $ \beta ^{1}_S$ need not be  suitable to form the twisted martingale differences; 
(B) even if defined, the twisted martingale differences associated to bad cubes need not vanish; and (C) the functions $ f_j$ are unbounded.  A truncation of the stopping
tree will address all of these three issues.

Concerning point (B),  %one should not fail to note that $ \Delta _Q ^{\beta^j } \mathbf{1}$ need not be zero; nevertheless, 
there is 
a simple sufficient condition for a twisted martingale difference to be identically zero.   

%%%%%%%%%%%%%%%%%%%%%%%%%%%%%% PROPOSITION PROPOSITION PROPOSITION
\begin{proposition}\label{p.D=0} Assume that
 $ Q \in\mathscr{A}^j$  is bad, and no child of $ Q$ is 
 in $ \widetilde{\mathcal S} ^{j}$. 
 Suppose  $\langle \beta^j_{S}\rangle_Q\not=0$, where $S=\pi_{\mathcal{\widetilde{S}}^j} Q$.
Then, 
both $ \Delta  ^{\beta ^j}_Q f _j $ and $ \widetilde{D}  ^{\beta ^j}_Q f _j $ are well defined   using  $\widetilde{\mathcal S}^j$ in parent
selectors for $\beta^j$ functions,
and $ \Delta _Q ^{\beta ^{j} } f_j \equiv 0 \equiv \widetilde{D}  ^{\beta ^j}_Q f _j$. 
\end{proposition}
%%%%%%%%%%%%%%%%%%%%%%%%%%%%%% PROPOSITION PROPOSITION PROPOSITION

%%%%%%%%%%%%%%%%%%%%%%%%%%%%%% PROOF PROOF PROOF
\begin{proof}
	By assumptions and definitions, the  averages of $ f_j$ and $ \beta ^{j} _{S}$ do not change  moving from cube $ Q$ to a child of $ Q$.  
	By inspection of \eqref{e.twist} and \eqref{e.halfDef}, the ratios in the definition of 
	either martingale difference of $f_j$ are all 
	well defined and equal, hence they cancel.  
\end{proof}
%%%%%%%%%%%%%%%%%%%%%%%%%%%%%% PROOF PROOF PROOF

The previous considerations  lead to the
following three types of undesirable cubes $ Q \in \mathcal A  ^{j}$,  where 
$\Lambda>1$ and $ 0 < \upsilon_1 <  4 ^{- 1 - n}$  are absorption parameters  and $ \{j,k\}= \{1,2\}$:
\begin{description}
\item[Type A] 
 %$\inf_{x\in Q} M|\widetilde \beta^j_{\pi  _{\widetilde{\mathcal S} ^{j}}Q} \rvert^{p_j}(x)  +
 \{$\langle \lvert  M\widetilde\beta ^{j} _{\pi  _{\widetilde{\mathcal S} ^{j}}Q }\rvert ^{ p_j}\rangle_Q  \ge \upsilon_1^{p_j}$ or
$ \langle \lvert  T^j\widetilde  \beta ^{j} _{\pi  _{\widetilde{\mathcal S} ^{j}}Q }
	\rvert ^{ p_k'}   \rangle_Q 
	\ge \upsilon_1^{p_k'} \mathbf T  ^{p_k'}   $ \}
	or
 	 $Q$ has a  child $S\in\widetilde{\mathcal{S}}^j$
	such that 
	\{$\langle \lvert M\widetilde \beta^j_S\rvert^{p_j} \rangle_S \ge \upsilon_1^{p_j}$
	or 
	$ \langle \lvert  T^j\widetilde  \beta ^{j} _{S}
	\rvert ^{p_k'}   \rangle_S 
	\ge \upsilon_1^{p_k'} \mathbf T  ^{p_k'}   $\};
\item[Type B]   $ Q$ is not of Type A and  $ Q$ has a child in $ \widetilde {\mathcal S}^{j}$, and $ Q$ is bad;
	  
\item[Type C]  $ Q$ is not of Type A, nor Type B, and $ \langle \lvert  f_j\rvert  \rangle_Q >  \Lambda  $. 
\end{description}
Each of these three types  depend upon both $ \omega ^{1}$ and $ \omega ^{2}$.
Let $ \mathcal B ^{j, \alpha}$ be the collection of  maximal  cubes in $\mathcal{A}^j$
of Type $ \alpha$, $ \alpha=A,B,C$, and let $ \mathcal B ^{j}$ be the maximal cubes  in the union of  these three collections.  
Define $ B ^{j, \alpha } := \bigcup \{Q \;:\; Q \in \mathcal B ^{j,\alpha }\}$, and $ B ^{j} := B ^{j,A}\cup B ^{j,B} \cup B ^{j,C}$.  
%It is tempting to restart the selection of the $ b$ functions for those  cubes $ Q\in \mathcal B ^{j}$, but this 
%would introduce opaque interdependicies among the  selected functions, since $ \mathcal B ^{j}$ is  a function of both $ \omega ^{1}$ and $ \omega ^{2}$.  
	%Instead, what we do 

Let us verify that the sets $ B ^{j}$ are small in measure, on average.
	Therefore certain error terms coming from the truncation can be later absorbed.

\begin{lemma}\label{e.EB<} For $\{j,k\}=\{1,2\}$, we have 
$	\mathbb E \lvert   B ^{j}\rvert=\mathbb{E}^j\mathbb{E}^k \lvert   B ^{j}\rvert  \lesssim \big\{ \Lambda ^{-2p_j} +  
	\upsilon_1 ^{-p_j}    2 ^{- c \epsilon r}\big\} \lvert  Q ^{0}\rvert$.
\end{lemma}

\begin{proof}	
We first prove that 
\begin{equation}\label{e.typeA<}	\mathbb{E}^k \lvert  B ^{j,A}\rvert \lesssim \upsilon_1 ^{- p_j} 2 ^{- c \epsilon r} 
	\lvert  Q ^{0}\rvert$, where $\{j,k\}=\{1,2\}\,.
	\end{equation}
Recall that the collection $\widetilde{\mathcal S} ^{j}$ is only a function of $ \omega ^{j}$.  
By sparsness,
$	\sum_{S\in \widetilde{\mathcal S} ^{j}} \lvert  S\rvert \lesssim \frac 1 {1 - \tau } \lvert  Q ^{0}\rvert \lesssim \lvert  Q ^{0}\rvert$.
A cube is of Type A for four potential reasons;
Fix $ S\in \widetilde{\mathcal{S}}^j$, and 
let $ \mathcal B ^{j,A_1} _S $ be the maximal cubes
$ Q\in\mathcal{A}^j$ with $ \pi _{\widetilde{\mathcal S} ^{j}} Q=S$, and
$ \langle \lvert M\widetilde  \beta ^{j} _S\rvert ^{ p_j}   \rangle_Q \ge  \upsilon_1^{p_j}$. 
By Lemma \ref{l.tildeB},   
\begin{equation*}
	\mathbb{E}^k \sum_{Q\in\mathcal B ^{j,A_1} _S } \lvert  Q\rvert 
	\lesssim \upsilon_1 ^{-p_j}\mathbb{E}^k \int _{S} \lvert  \widetilde \beta ^{j}_S\rvert ^{p_j} 
	\lesssim \upsilon_1 ^{-p_j} 2 ^{- c \epsilon r} \lvert  S\rvert \,.  
\end{equation*}
Second, let $ \mathcal B ^{j,A_2} _S $ be the maximal cubes $ Q\in \mathcal{A}_j$ with $ \pi _{\widetilde{\mathcal S} ^{j}} Q=S$, and
$ \langle \lvert T^j\widetilde  \beta ^{j} _S\rvert ^{ p_k'}   \rangle_Q \ge  \mathbf T ^{p_k'}  \upsilon_1^{p_k'}$.   
Then, using the a priori norm
bound $c\mathbf{T}$ for the operator $T^j$ on $L^{p_k'}$ and inequality $p_k'\le p_j$, 
\begin{equation*}
	\mathbb{E}^k \sum_{Q\in\mathcal B ^{j,A_2} _S } \lvert  Q\rvert 
	\lesssim \upsilon_1 ^{-p_k'} \cdot \mathbb{E}^k \int _{S} \lvert  \widetilde \beta ^{j}_S\rvert ^{p_k'} 
	\lesssim \upsilon_1 ^{-p_j} 2 ^{- c \epsilon r}  \lvert  S\rvert \,.  
\end{equation*}
 Third, let $\mathcal{B}^{j,A_3}$ be the collection of cubes $Q$ in
$\mathcal{A}^j$, having a child
$S\in\widetilde{\mathcal{S}}^j$ with
$\langle \lvert M\widetilde{\beta}^j_S \rvert^{p_j}\rangle_S \ge \upsilon_1^{p_j}$.
Then,
\begin{align*}
\mathbb{E}^k\sum_{Q\in\mathcal{B}^{j,A_3}} \lvert Q\rvert 
\lesssim  \upsilon_1^{-p_j} \sum_{S\in\widetilde{\mathcal{S}}^j} 
\mathbb{E}^k\int_S \lvert
\widetilde\beta^j_S\rvert^{p_j} \lesssim \upsilon_1 ^{-p_j} 2 ^{- c \epsilon r} 
\sum_{S\in \widetilde{\mathcal{S}}^j}\lvert  S\rvert \lesssim \upsilon_1 ^{-p_j} 2 ^{- c \epsilon r} \lvert Q^0\rvert\,.
\end{align*}
A similar estimate for the remaining collection
$\mathcal{B}^{j,A_4}$ of cubes $Q$ in
$\mathcal{A}^j$, having a child
$S\in\widetilde{\mathcal{S}}^j$ such that
$\langle \lvert T^j\widetilde{\beta}^j_S \rvert^{p_k'}\rangle_S \ge \upsilon_1^{p_k'}\mathbf{T}^{p_k'}$,
finishes the proof of inequality \eqref{e.typeA<}. 
	
Let us then consider the set  $  B ^{j,B} $.
The collection $\widetilde{\mathcal S} ^{j}$ is only a function of $ \omega ^{j}$, and holding that variable fixed, 
the event that $ S\in \widetilde{\mathcal S} ^{j}$ has a bad parent 
is an event in $ \Omega ^{k}$. 
And so, 
\begin{equation}\label{e.typeB<}
	\mathbb{E}^k \lvert  B ^{j,B}\rvert  \le 2^n\cdot \mathbb{E}^k
	\sum_{S\in \widetilde{\mathcal S} ^{j}} \lvert  S\rvert \mathbf 1_{ \pi S \textup{ is bad}} 
	\lesssim \frac 1 { 1- \tau } 2 ^{- \epsilon r} \lvert  Q ^{0}\rvert  \lesssim  2 ^{- \epsilon r} \lvert  Q ^{0}\rvert\,.   
\end{equation}
For the remaining set $ B ^{j,C}$, recall that $ f_j $ is a dyadic 
$ BMO$ function, uniformly over $\omega^1$ and $\omega^2$. More precisely, by 
Chebyshev's inequality and \eqref{e.f_bmo}, we have 
\begin{equation}\label{e.typeC<}
	|B^{j,C}|= \sum_{Q\in \mathcal B ^{j,C}} \lvert  Q\rvert \le \lvert  \{ M f_j > \Lambda \} \rvert 
		 \le \Lambda ^{-2p_j} \lVert Mf_j\rVert_{2p_j}^{2p_j} \lesssim \Lambda^{-2p_j}\lvert Q^0\rvert\,.   
\end{equation}
The proof is completed by combining inequalities \eqref{e.typeA<}, \eqref{e.typeB<},
and \eqref{e.typeC<}.
\end{proof}

		Next we define the collection $ \mathcal G ^{j}$, and the 
 perturbed 
stopping data for $\mathcal{G}^j$, claimed by Lemma~\ref{l.GB}. 
	This is done by truncating the stopping tree $ \widetilde {\mathcal S} ^{j} $ at  $ \mathcal B ^{j}$.

\begin{definition}
Take $ \mathcal G ^{j}$ to be all \emph{good} cubes in $ \mathcal A ^{j}$ that are \emph{not contained} in  any cube in $\mathcal{B} ^{j} $.
Set  $  {\mathcal S} ^{j}$ to be $  \widetilde {\mathcal S }^{j}$ minus all cubes
that are \emph{strictly} contained in some $ Q\in \mathcal B ^{j}$.  For convenience,
 we also denote by $\mathcal{R}^j\supset \mathcal{G}^j$ all cubes in $\mathcal{A}^j$,
  {\em both good and bad},  not contained in any cube in
 $\mathcal{B}^j$.  Take the 
  data for $ \mathcal G ^{j}$ to be $ \{ \beta ^{j}_S \;:\; S\in {\mathcal S} ^{j}\}$.
  \end{definition}

   Let us emphasize the fact that $Q\in\mathcal{R}^j$ is not
 of any Type $\alpha$, $\alpha=A,B,C$. 
In the remaining part of this section, we will check all the assertions in Lemma \ref{l.GB}.

\subsection*{Verification of the Perturbed Stopping Data}
First we show that $ \{ \beta ^{j}_S \;:\; S\in {\mathcal S} ^{j}\}$ is indeed a perturbed stopping data
for $\mathcal{G}^j$, as claimed.
By construction,
\begin{equation}\label{e.coincide}
\pi_{\mathcal{S}^j} Q = \pi_{\widetilde{\mathcal{S}}^j} Q,\quad \pi_{\mathcal{S}^j} Q' = \pi_{\widetilde{\mathcal{S}}^j} Q'
\end{equation}
if  $Q\in\mathcal{R}^j$  and $Q'\in\textup{ch}(Q)$.
Accordingly $ \{ \beta ^{j}_S \;:\; S \in \mathcal S ^{j}\}$ satisfies property (1) 
in Definition~\ref{d.stoppingData} of 
 perturbed  stopping data.  
 Another consequence of \eqref{e.coincide} is
 that we can compute the martingale differences 
 $\Delta_P^{\beta^j}$ and $\widetilde{D}_P^{\beta^j}$
in case of $P\in\mathcal{R}^j$
 by using freely either $\mathcal{S}^j$ or $\widetilde{\mathcal{S}}^j$
 in the parent selectors for $\beta^j$ functions.

The sparseness property (3) is trivial 
for $\mathcal{S}^j$,
since
$\widetilde{\mathcal{S}}^j$ satisfies it and $\mathcal{S}^j\subset \widetilde{\mathcal{S}}^j$.
The remaining properties (2a)---(2c) of the perturbed stopping data follow from the next lemma.

%%%%%%%%%%%%%%%%%%%%%%%%%%%%%% LEMMA LEMMA LEMMA
\begin{lemma}\label{l.typeA}  Fix $j\in \{1,2\}$ and a cube $ S \in \mathcal S ^{j}$. 
	Then, the following conditions (1)---(3) hold:
	%%  ENUMERATE
\begin{enumerate}
\item  $ \langle \beta ^{j}_S \rangle_S = 1$;
	\item $\langle \lvert  \beta ^{j}_S\rvert ^{p_j}  \rangle_S  \lesssim \mathbf A ^{p_j} $;
	\item  Suppose $Q\in \mathcal{R}^j$ and		
		$ \pi _{\mathcal S ^{j}}Q=S$ (or
			$Q$ is a child of a cube in $\mathcal{R}^j$ and		
		$ \pi _{\mathcal S ^{j}}Q=S$). Then
	%%  ENUMERATE
\begin{enumerate}
\item  $  \langle \beta ^{j}_S \rangle_Q \ge \tfrac 1 4 $;
	
\item $ \langle \lvert   M\beta ^{j}_S\rvert ^{p_j}  \rangle_Q  \lesssim  \delta ^{-1}\mathbf{A}^{p_j}$;

\item $ \langle \lvert T^j  \beta ^{j}_S\rvert ^{ p_k'}  \rangle_Q  \lesssim  \delta ^{-1} \mathbf T ^{p_k'} _{\textup{loc}} 
	+ \upsilon _1 ^{p_k'} \mathbf T ^{p_k'}$, where $\{j,k\}=\{1,2\}$.
\end{enumerate}
%% ENUMERATE	
\end{enumerate}
%% ENUMERATE
\end{lemma}
%%%%%%%%%%%%%%%%%%%%%%%%%%%%%% LEMMA LEMMA LEMMA

%%%%%%%%%%%%%%%%%%%%%%%%%%%%%% PROOF PROOF PROOF
\begin{proof}
	By Definition \eqref{e.beta},  $\lvert  S\rvert=  \int _{S} b ^{j}_S \; dx = \int _{S} \beta ^{j}_S \; dx $, 
	so 
	property (1) holds. The boundedness of martingale transforms implies property (2): 
		$\int_S \lvert \beta ^{j}_S  \rvert  ^{p_j} \; dx \lesssim \int _{S } \lvert  b ^{j}_S\rvert ^{p_j} \; dx \le
		 \mathbf A ^{p_j}\lvert S\rvert \,. $

The properties (3a)---(3c) are a consequence  of equation \eqref{e.coincide}   and the failure of condition defining Type A cubes.   Let us first consider property (3a).
If $Q\in\mathcal{R}^j$ and $\pi_{\mathcal{S}^j}Q=S$, then
\begin{equation*}
	\langle \beta ^{j}_S \rangle_Q \ge \langle b ^{j}_S \rangle_Q - 
	\langle \lvert M\widetilde  \beta ^{j} _{S}\rvert^{p_j}  \rangle_Q^{1/p_j} \ge \tfrac 12 - \upsilon_1\,,
\end{equation*}
which is greater than $ 1/4$ (Recall that {\em stopping data} is
slightly stronger on this point).
	If $Q$ is a child of a cube in $\mathcal{R}^j$
	and $\pi_{\mathcal{S}^j}Q=S$, then
	either $Q\in\mathcal{S}^j$, in which case $\langle \beta ^{j}_S \rangle_Q=\langle \beta ^{j}_Q \rangle_Q=1$,
or
        the property  (3a) follows as above by first
        comparing
	the average of $\lvert \widetilde  \beta ^{j} _{S}|$ on $Q$ to its average on  $\pi Q$.
		Let us then consider (3b) and (3c) for $Q\in\mathcal{R}^j$.
	By sub-linearity  and stopping rules,
\begin{equation*}
 \langle \lvert M  \beta ^{j}_S\rvert ^{p_j}  \rangle_Q  
 \lesssim  \langle \lvert M  b ^{j}_S\rvert ^{p_j}  \rangle_Q  + 
  \langle \lvert M \widetilde  \beta ^{j}_S\rvert ^{p_j}  \rangle_Q  
  \le \delta ^{-1} \mathbf{A}^{p_j} + \upsilon_1^{p_j}\lesssim \delta ^{-1} \mathbf{A}^{p_j} \,. 
\end{equation*}
Likewise, $	\langle \lvert T^j  \beta ^{j}_S\rvert ^{p_k'}  \rangle_Q  
 \lesssim  \langle \lvert T^j  b ^{j}_S\rvert ^{p_k'}  \rangle_Q  + 
 \langle \lvert T^j \widetilde  \beta ^{j}_S\rvert ^{p_k'}  \rangle_Q  
 \le \delta ^{-1} \mathbf{T}^{p_k'} _{\textup{loc}} + \upsilon_1^{p_k'}\mathbf{T}^{p_k'}$.
These properties for a child $Q$ of a cube in $\mathcal{R}^j$ follow by comparing the average on $Q$ to that on $\pi Q$, in case of $Q\not\in\mathcal{S}^j$, and
by the stopping rules in case of $Q\in\mathcal{S}^j$. 
\end{proof}
%%%%%%%%%%%%%%%%%%%%%%%%%%%%%% PROOF PROOF PROOF

\subsection*{Verification of Conditions (1)---(3) in Lemma~\ref{l.GB}}
Every cube $ Q\in \mathcal G ^{j}$ is good by definition and,
by construction, $\langle |f_j|\rangle_Q\le \Lambda$ (recall Type C cubes).
Let us then consider the property (3), concerning the sum 
of  half-twisted differences in \eqref{e.BNDD}.  
For a fixed  $ Q \in \mathcal G ^{j}$ with a child $ Q'$, and $ S\in \mathcal S  ^{j}$ with  $ \pi _{\mathcal S ^{j}} Q \subset S$,  %there holds $ \lvert  \lambda _{Q'}\rvert  \lesssim   \Lambda    $, where $\lambda_{Q'}$ is
let us consider the constant $\lambda_{Q'}$ defined by 
\begin{equation*}
	\lambda _{Q'}\mathbf{1}_{Q'}:=\mathbf 1_{Q'} \sum_{\substack{ P\in\mathcal{A}^j  \;:\; P\supset Q\\  \pi _{\mathcal S ^{j}}P= S}}
	\widetilde D _P^{\beta^j} f_j\,.
\end{equation*}
In contrast to the series in \eqref{e.BNDD}, 
the series above extends over \emph{all cubes} with the same $ \mathcal S ^{j}$ parent.  
Nevertheless, 
we are not redefining $\lambda _{Q'}$. Indeed,
if $ P $ is a bad cube in the series above, 
then $P\in\mathcal{R}^j$ and it has no stopping children in $\widetilde{\mathcal{S}}^j$ due to the construction;
by property
\eqref{e.coincide} and Proposition~\ref{p.D=0}, we find that
 $\widetilde D _P^{\beta^j} f_j\equiv 0$, so the two series, in fact, coincide.

Then, by inspection of \eqref{e.halfDef}, 
the series above on $Q'$ is telescoping to the difference of two ratios
 (or to a single ratio). 
 On the numerator of the ratios are averages of $ f_j$, which 
are bounded by  the definition of Type C cubes.  The denominator of the ratios is an average of $ \beta ^{j}_S$, which is bounded below 
by $ 1/4 $ because of  (3a) in Lemma \ref{l.typeA}.
All in all, we find that $  \lvert  \lambda _{Q'}\rvert  \lesssim  \Lambda  $.
%, it follows that the child of $ P$ that contains $ Q'$ must be a stopping cube, hence cannot contribute to the sum.  
 \qed

%%%%%%%%%%%%%%%%%%%%%%%%%%%%%% SUBSECTION SUBSECTION SUBSECTION SUBSECTION
%%%%%%%%%%%%%%%%%%%%%%%%%%%%%% SUBSECTION SUBSECTION SUBSECTION SUBSECTION 
\subsection{Completion of the proof of Lemma~\ref{l.GB}} 
The proof of inequality \eqref{e.goodSum}  remains,
% This inequality implies that it suffices to bound the 
%sum over the collection $ \mathcal G ^{1} \times \mathcal G ^{2}$,
%compare to inequality \eqref{e.GS<}.
and we need an appropriate representation
formula for $f_j$'s, so that we can compute the difference in \eqref{e.goodSum}.
We begin with certain preparations for the representation Lemma \ref{l.representation}.

Define $ \phi ^{j} := \sum_{Q\in \mathcal B ^{j}} \phi ^{j}_Q $,  where $\phi_Q^j = f_j \mathbf{1}_Q$ if $Q\in \mathcal{B}^j\cap \mathcal{A}^j_*$ and, otherwise, 
\begin{equation} \label{e.phiDef}
	\phi _{Q} ^{j} := 
	f_j \mathbf 1_{Q} - \frac { \langle f_j \rangle_Q} {\langle  \beta ^{j} _{\pi _{\mathcal S^j}Q}\rangle_{Q}} 
	\beta ^{j} _{\pi _{\mathcal S ^{j}}Q } \mathbf 1_{Q}\,.
\end{equation}
	For the following lemma, recall that the set $ B ^{j}$ is a function of both $\omega  ^{1}$ and $\omega ^{2} $, and it is of small measure in expectation. 

\begin{lemma}\label{e.phi_est}
We have 
$\lVert \phi ^{j}\rVert_{p_j} ^{p_j}  \lesssim  \Lambda^{p_j}  \lvert  B ^{j}\rvert$ for $j\in \{1,2\}$.
\end{lemma}

\begin{proof}
%
%	Let us prove inequality \eqref{e.phi_est}.
If $Q\in \mathcal{B}^j\cap \mathcal{A}^j_*$ then, by \eqref{e.f_bmo},
$\lVert \phi_Q^j\rVert_{p_j} \lesssim \rvert Q^0\rvert^{1/p_j}\lesssim |B^j|^{1/p_j}$.
There are at most $2^n$ such cubes.  
For the remaining terms we notice that, since $ f _{j}$ is in BMO and the average values of $ f_j$ are controlled, 
\begin{align} 
	\Big\lVert \sum_{Q\in \mathcal B ^{j}\setminus \mathcal{A}^j_*} \phi ^{j}_Q\Big\rVert_{p_j} ^{p_j} 
	&= \sum_{Q} 
	\Bigl\lVert 
		f_j \mathbf 1_{Q} - \frac { \langle f_j \rangle_Q} {\langle  \beta ^{j} _{\pi _{\mathcal S^j}Q}\rangle_{Q}}  \beta ^{j} _{\pi _{\mathcal S ^{j}}Q } \mathbf 1_{Q} 
		\Bigr\rVert_{p_j} ^{p_j}
		\\  \label{e.zf<}
		& \lesssim   
		 \sum_{Q} 
	\biggl\{\lVert 
	f_j \mathbf 1_{Q} -  \langle f_j \rangle_Q \mathbf 1_{Q} \rVert_{p_j} ^{p_j} 
	+ \Bigl\lVert   \langle f_j \rangle_Q \mathbf 1_{Q} -  \frac { \langle f_j \rangle_Q} {\langle  \beta ^{j} _{\pi _{\mathcal S^j}Q}\rangle_{Q}}  \beta ^{j} _{\pi _{\mathcal S ^{j}}Q } \mathbf 1_{Q} 
	\Bigr\rVert_{p_j} ^{p_j} \biggr\}
\\		& \lesssim   \Lambda^{p_j} 
 \sum_{Q\in \mathcal B ^{j}\setminus \mathcal{A}^j_\ast}  \lvert  Q\rvert  \le  
 \Lambda^{p_j}   \lvert  B ^{j}\rvert  \,. 
		\end{align}
		 We used definition of Type C cubes and Lemma \ref{l.typeA}, along with the
		observation that the parent of $Q$ is in $\mathcal{R}^j$ if $Q\in\mathcal{B}^j\setminus\mathcal{A}^j_\ast$. 				\end{proof}

	Concerning the representation of $ f_j$, we have the following lemma.

%%%%%%%%%%%%%%%%%%%%%%%%%%%%%% LEMMA LEMMA LEMMA
\begin{lemma}\label{l.representation} 
Fix $j\in \{1,2\}$. Then, the following equality holds almost everywhere and in $ L ^{p_j}$
	\begin{equation} \label{e.representation}
	f_j  =  
	\sum_{Q\in \mathcal A ^{j} _{\ast }\setminus \mathcal{B}^j} \langle f_j \rangle_Q \beta ^{j}_Q + 
	\sum_{Q \in \mathcal G ^{j} } \Delta_Q ^{\beta ^{j}} f_j  + \phi ^{j}  \,. 
\end{equation}
\end{lemma}
%%%%%%%%%%%%%%%%%%%%%%%%%%%%%% LEMMA LEMMA LEMMA

%%%%%%%%%%%%%%%%%%%%%%%%%%%%%% PROOF PROOF PROOF
\begin{proof}
	Let $ Q$ be any bad cube, which is not contained in a cube in $\mathcal{B}^j$.
	By construction and Proposition~\ref{p.D=0}, $ \Delta ^{\beta ^{j}}_Q f _j \equiv 
	0$.  It follows that for any $ x\in B^j$, the sum above is in fact finite, and telescoping.  By inspection, it is equal to $ f_j (x)$.
	
	Consider $ x \not\in B^j$. Then by Proposition \ref{p.D=0}, for any cube $ P \ni x$,   
	\begin{equation*}
	\sum_{Q\in \mathcal A ^{j} _{\ast } \setminus \mathcal{B}^j} \langle f_j \rangle_Q \beta ^{j}_Q (x) + 
	\sum_{Q \in \mathcal G ^{j} \;:\; P\subsetneq Q  } \Delta_Q ^{\beta ^{j}} f_j (x) 
	=  \frac {\langle f_j \rangle_P} { \langle \beta  ^{ j} _{\pi _{\mathcal S_j} P} \rangle_P} \beta ^{j} _{\pi _{\mathcal S^j} P} (x) \,. 
\end{equation*}
Now, since $ \mathcal S ^{j}$ is sparse, almost every $ x $ is in only a finite number of cubes $ S \in \mathcal S^j$.  
%Hence, for all $ P$ sufficiently small,  $ \pi _{\mathcal S^j} P =S$ does not change,
Hence,  the proof is finished by appealing to a straightforward modification of \cite{1011.0642}*{Lemma 3.5}.
%And, by the martingale convergence theorem, for almost every $ x\in S$, 
%$
%	\langle \beta  ^{ j} _{S} \rangle_P \to \beta ^{j} _{S} (x) 
%	\quad \textup{and}\quad  \langle f_j \rangle_P \to f_j (x) \,. $
%The convergence is as $ \ell P \to 0$, with $ x \in P$.  The convergence is in $ L ^{p_j}$. Moreover, 
%by the stopping rules, the averages of $ \beta ^{j}_S$ are bounded from above and below, so that $ L ^{p_j}$ 
%convergence also holds for the inverses as well.  
\end{proof}
%%%%%%%%%%%%%%%%%%%%%%%%%%%%%% PROOF PROOF PROOF

 We also need a Hardy inequality.
 For a proof, we refer to \cite{1011.1747}*{Section 9}.
 
 %%%%%%%%%%%%%%%%%%%%%%%%%%%%%% LEMMA LEMMA LEMMA
 \begin{lemma}\label{l.hardy}  Let $Q$ be  any cube in $\mathbf{R}^n$ and $\kappa>1$. 
 	 For every $ 1< p < \infty $, there holds
 	 \begin{equation}\label{e.hardy}
 	 	  \int _{\kappa Q \setminus Q} \int _{Q} \frac {\lvert  g_1 (y) g_2 (x)\rvert  } { \lvert  x-y\rvert ^{n} } \; dy dx 
 	 	 \lesssim \lVert g_1\rVert_{p} \lVert g_2\rVert_{p'} \,,\quad 1/p+1/p'=1\,.
\end{equation}
The implied constant depends upon $\kappa, p, n$.
\end{lemma}
%%%%%%%%%%%%%%%%%%%%%%%%%%%%%% LEMMA LEMMA LEMMA

 \begin{proof}[Proof of inequality \eqref{e.goodSum}]
When expanding $ \langle T f_1, f_2 \rangle$ by using \eqref{e.representation}, there are a number of error terms. 
They are treated by the following estimates, and their duals, as applicable, which we do not directly state.   
For $ P\in \mathcal A ^{1} _{\ast }\setminus \mathcal{B}^1$, 
the cubes $P$ and $Q^0$ are roughly of the same size, so
that $\lvert \langle f_1\rangle_P\rvert \lesssim 1$
by inequality \eqref{e.f_bmo}. 
Furthermore, using the local $ Tb$ hypothesis, definition
of Type A cubes, and the Hardy inequality stated in Lemma \ref{l.hardy},
\begin{align*}
	\lvert  \langle T \beta ^{1} _{P}, f_2\rangle\rvert & \le \lvert  \langle T \beta ^{1} _{P}, f_2 \mathbf 1_{P}\rangle \rvert + 
	\lvert \langle  T \beta ^{1} _{P}, f_2 \mathbf 1_{ 6P \backslash P }\rangle\rvert
	\lesssim \{1 + \mathbf T _{\textup{loc}} + \upsilon_1 \mathbf{T}\} \lvert  Q ^0\rvert\,.  
\end{align*}
And for $ P \in \mathcal A ^{1} _{\ast }\setminus \mathcal{B}^1$, and 
$ Q\in \mathcal A ^{2} _{\ast }\setminus \mathcal{B}^2$, likewise, we have 
$	\lvert  \langle T \beta ^{1} _{P}, \beta ^{2}_Q \rangle \rvert  \lesssim  
	\{1 + \mathbf T _{\textup{loc}} + \upsilon_1 \mathbf{T} \} \lvert  Q^0\rvert\,.  
$
Next, for  a cube $P$ as above, there holds by the assumed norm inequality on $ T$, Lemma \ref{e.phi_est} and Lemma \ref{e.EB<},
\begin{align*}
\mathbb{E}\big\{  \lvert \langle T \beta_P^1, \phi^2\rangle\rvert + \lvert \langle T\phi^1, f_2\rangle\rvert  \big\}
\lesssim 
	\mathbf{T}\cdot \{ \Lambda ^{-1} +  
	\Lambda\upsilon_1 ^{-1}\cdot 2 ^{- c \epsilon r} \}\cdot  \lvert  Q ^{0}\rvert\,.
\end{align*}
Lastly, when
$\Lambda^{-1}+\Lambda\upsilon_1^{-1}2^{-c\epsilon r}<1$,  we have
$\mathbb{E}\lvert \langle T\phi^1,\phi^2\rangle\rvert
\lesssim 
\mathbf{T}\cdot \{ \Lambda ^{-1} +  
	\Lambda\upsilon_1 ^{-1}\cdot 2 ^{- c \epsilon r} \}\cdot  \lvert  Q ^{0}\rvert\,.$
When combined with
\eqref{e.f_bmo},
these inequalities---and their duals---complete the proof of \eqref{e.goodSum}. \end{proof}

The proof of Lemma~\ref{l.GB}, and the corona construction, are both complete.
%%%%%%%%%%%%%%%%%%%%%%%%%%%%%% SECTION  SECTION SECTION
%%%%%%%%%%%%%%%%%%%%%%%%%%%%%% SECTION  SECTION SECTION 
\section{Useful Inequalities}\label{s.inequalities}
%%%%%%%%%%%%%%%%%%%%%%%%%%%%%% SUBSECTION SUBSECTION SUBSECTION SUBSECTION
 %%%%%%%%%%%%%%%%%%%%%%%%%%%%%% SUBSECTION SUBSECTION SUBSECTION SUBSECTION 
 \subsection{The Martingale Transform Inequalities}\label{ss.martingale}

 We recall essential tools that we will need.  
 Fix a function $ b$ supported on a dyadic\footnote{
In our applications, the underlying dyadic grid will be $\mathcal{D}^j$, $j\in \{1,2\}$.}
 cube $ S_0$, satisfying $ \int b \; dx = \lvert  S_0\rvert $ and 
$ \lVert b\rVert_{p} \le \mathbf B\lvert  S_0\rvert ^{1/p} $, where $ 1<p< \infty $ is fixed. 
We will consider a fixed but arbitrary collection $\mathcal{T}$ of
disjoint dyadic cubes inside $S_0$, the `terminal cubes'.
Let $ \mathcal Q$ be all dyadic cubes, contained in $S_0$, but not contained in any terminal cube $ T\in \mathcal T$.
We require that there is $\sigma\in(0,1)$ such that, for all $Q\in\mathcal{Q}$,
\begin{equation}\label{e.ra}
	\Bigl\lvert \int _{Q} b \; dx  \Bigr\rvert \ge 4^{-1}  \lvert  Q\rvert\, \quad \textup{and} \quad  
%	\Bigl\lvert \int _{T} b \; dx  \Bigr\rvert \ge \delta^{-1}   \lvert  T\rvert\, \quad \textup{or} \quad
%	 \Bigl\lvert \int _{T} b_P \;dx \Bigr\rvert > \delta ^{-1} \mathbf A \lvert  T\rvert\,, \quad \textup{or} \quad  
	\int _{Q} \lvert b\rvert ^{p}\,dx \le \sigma^{-1}  \mathbf{B}^{p}\lvert  Q\rvert\,.  
\end{equation}
For each terminal cube $ T$, we have a function $ b_T$ supported on $ T$, and  satisfying 
$\int b_T\; dx=|T|$ and $ \lVert b_T\rVert_{p} \le \mathbf{B}|T|^{1/p}$.
If the conditions above are met, then we say that
the collection, comprised of functions $b$ and $b_T$, $T\in\mathcal{T}$, is {\em admissible}.
We will not keep track of the constants $\sigma$ and $\mathbf{B}$, and the
implied constants will depend upon them.

 For $Q\in \mathcal Q$ we define 
 the \emph{(half) twisted martingale differences}
\begin{align*}
	D_{Q} ^{b} f & := 
\sum_{Q'\in \textup{ch}(Q)\setminus \mathcal T} 
	\biggl\{
	\frac {\langle f \rangle _{Q'}} {\langle b \rangle _{Q'}   } 
	-
	\frac {\langle f \rangle _{Q}} {\langle b \rangle_Q}  
	\biggr\}  \mathbf 1_{Q'} \,, 
\\ 
	\widetilde D_{Q} ^{b} f & := 
\biggl\{ \sum_{Q'\in \textup{ch}(Q)\setminus \mathcal T} 
	\frac {\langle f \rangle _{Q'}} {\langle b \rangle _{Q'}   } \mathbf{1}_{Q'} \biggr\}
	-
	\frac {\langle f \rangle _{Q}} {\langle b \rangle_Q}  \mathbf{1}_Q
	 \,, 
\\ 
\Delta _{Q}  ^{b} f & := 
		\sum_{Q'\in \textup{ch}(Q)} 
	\biggl\{
	\frac {\langle f \rangle _{Q'}} {\langle b_{Q'} \rangle _{Q'} } b _{Q'}
	-
	\frac {\langle f \rangle _{Q}} {\langle b \rangle_Q}  b 
	\biggr\}  \mathbf 1_{Q'} \,, 
\end{align*}
where we set $ b _{Q'} = b$ if $ Q'\not\in \mathcal T$ and otherwise, 
$ b_{Q'}$ is defined as above. 

The following theorem  is proved in \cite{1011.1747}*{Lemma 5.3} and  \cite{lv-perfect}*{Section 2}.

%%%%%%%%%%%%%%%%%%%%%%%%%%%%%% THEOREM THEOREM THEOREM
\begin{theorem}\label{t.mt} 
 Suppose that $b$ and $b_T$, $T\in\mathcal{T}$, constitutes
an admissible collection. Then, the following inequalities hold for all selections of constants $ \lvert  \varepsilon _Q\rvert \le1$
indexed by $Q\in \mathcal{Q}$:
	\begin{equation}
	\begin{split}
		&\Bigl\lVert  \sum_{Q\in \mathcal Q} \varepsilon _Q  \widetilde D _Q ^{b} f \Bigr\rVert_{q} + \Bigl\lVert  \sum_{Q\in \mathcal Q} \varepsilon _Q  D _Q ^{b} f \Bigr\rVert_{q} & \lesssim   \lVert f\rVert_{q} 
	\,, \qquad f\in L^q,\quad1< q < \infty \,,
	\\
	&\Bigl\lVert  \sum_{Q\in \mathcal Q} \varepsilon _Q \Delta _Q ^{b} f \Bigr\rVert_{p}  \lesssim   \lVert f\rVert_{p}\,, \quad f\in L^p\,,
\end{split}
\end{equation}
where $1<p<\infty$ is the exponent associated with the 
admissible function $b$.
\end{theorem}
%%%%%%%%%%%%%

 We will recourse to the following theorem several times.
Aside from Theorem \ref{t.mt}, it depends upon the sparseness of the stopping tree $ \mathcal S ^{j}$.% See  also the end of \cite{lv-perfect}*{Section 4}. 
 %%%%%%%%%%%%%%%%%%%%%%%%%%%%%% THEOREM THEOREM THEOREM
 \begin{theorem}\label{t.twisted} 
Fix $j\in \{1,2\}$. For each cube $Q$ 
 in $\mathbf{R}^n$, and any selection of coefficients 
 $\lvert \varepsilon _P\rvert \lesssim 1 $,
 	 \begin{equation} \label{e.twisted}
 	 	 \Bigl\lVert  \sum_{\substack{P\in\mathcal{G}^j \;:\;  P\subset Q }}  \varepsilon _P
 	 	 \Delta  _{P} ^{\beta ^{j}} f_j\Bigr\rVert_{p_j}
		  \lesssim \lvert  Q\rvert ^{1/p_j} \,.  
\end{equation}
The same statement holds true also with $\Delta  _{P} ^{\beta ^{j}}$ replaced by $\Delta  _{P} ^{b^{j}}$.
\end{theorem}
%%%%%%%%%%%%%%%%%%%%%%%%%%%%%% THEOREM THEOREM THEOREM

Before the proof of this theorem, let us make the following instructive remark.

\begin{remark}\label{r.mt}
Of particular importance in the sequel will be the following assignments.
For a fixed $S_0\in\mathcal{S}^j$ 
that is not contained in a cube in $\mathcal{B}^j$, we
set
$\mathcal{T}\subset \mathcal{D}^j$ to be  maximal cubes
in the collection
\[
\textup{ch}_{\mathcal{S}^j}(S_0) \cup \{T\,:\, T\subset S_0,\, T\in \textup{ch}(R),\,R\in\mathcal{B}^j\}\,.
\]
By construction of our perturbed stopping data, it is straight forward to verify that
the assignments $\beta := \beta_{S_0}^j$ and
\[
\beta_T := 
\begin{cases}
b_T^j \quad T\in \textup{ch}(R)\text{ for some }R\in\mathcal{B}^j\\
\beta_T^j\quad \text{otherwise}
\end{cases}
\]
yields an admissible collection, with $p=p_j$ 
and constants $\sigma\simeq \delta$ and $\mathbf{B}\simeq \mathbf{A}$.
Likewise, setting $b:= b_{S_0}^j$ and
$b_T:=b_T^j$ if $T\in\mathcal{T}$
yields admissible functions. Observe also that 
$P\in\mathcal{Q}$ if
$P\in\mathcal{G}^j$ satisfies $\pi_{\mathcal{S}^j} P =S_0$. 
Moreover, under the same assumption, 
$\Delta_P^{\beta} = \Delta_P^{\beta^j}$ and $\Delta_P^b=\Delta_P^{b^j}$.
Here, the right hand sides
are defined in \eqref{e.twist}.
Observe that the terminal functions $\beta_T$ and $b_T$ for
 $T\in \mathcal{T}\cap \textup{ch}(R)$, $R\in\mathcal{B}^j$,
do not play any role in these last identities.
\end{remark}

%%%%%%%%%%%%%%%%%%%%%%%%%%%%%% PROOF PROOF PROOF
\begin{proof}[Proof of Theorem \ref{t.twisted}]
By considering the disjoint collection of those maximal cubes in $\mathcal{G}^j$, that are
contained in $Q$, we are
reduced to the case of $Q\in\mathcal{G}^j$.
By Theorem \ref{t.mt}
and Remark \ref{r.mt},
we first obtain a weaker inequality. Indeed, letting $ S= \pi _{\mathcal S ^{j} } Q$, we have
	\begin{equation}\label{e:last}
 	 	 \Bigl\lVert  \sum_{\substack{P \;:\; \pi _{\mathcal S ^{j}} P = S\\ P\subset Q }} 
 	 	 \varepsilon _P	 \Delta  _{P}^{\beta^j} f_j\Bigr\rVert_{p_j} \lesssim
 	 	 \lVert   f_j \mathbf 1_{Q}\rVert_{p_j} \lesssim \lvert  Q\rvert ^{1/p_j}\,.
\end{equation}
We have the last inequality due to the construction of functions $ f_j$:---compare to inequalities in \eqref{e.zf<} and recall
the normalization of $f_j$ by $\Lambda^{-1}$.

We apply inequality \eqref{e:last} recursively for the remaining
terms, for which  $\pi_{\mathcal{S}^j} P\subsetneq Q$.
Let $ \mathcal R_{1}$ be the maximal $ R\in \mathcal S ^{j}$ strictly contained in $Q$, 
and inductively set $ \mathcal R_{k+1}$ to be the maximal cubes $ R'\in \mathcal S ^{j}$ strictly contained in any
$ R \in \mathcal R_{k}$.  By sparsness of $\mathcal{S}^j$,
\begin{equation*}
	\sum_{R\in \mathcal R_{k+1} } \lvert  R\rvert \le \tau  \sum_{R\in \mathcal R _{k}} \lvert  R\rvert
	\le \cdots \le  \tau ^{k} \lvert  Q\rvert\,, \qquad k\ge 1\,,
\end{equation*}
where $ 0< \tau < 1$.
Thus, setting $ \phi  _{k} := \sum_{R\in \mathcal R _{k}}\sum_{P \;:\; \pi _{\mathcal S ^{j}}P=R}  \varepsilon _P\Delta_{P}^{\beta^j} f_j$, 
there holds 
\begin{align*}
	\Bigl\lVert \sum_{k=1} ^{\infty } \phi_k \Bigr\rVert_{p_j} 
	= 	\Bigl\lVert \sum_{k=1} ^{\infty } k ^{1-1} \phi_k \Bigr\rVert_{p_j} 
	 \lesssim 
	\Bigl[ \sum_{k=1} ^{\infty } k ^{-p_j'} \Bigr] ^{1/p_j'} 
	\Bigl[\sum_{k} k ^{p_j}  \lVert \phi_k\rVert_{p_j} ^{p_j} \Bigr] ^{1/p_j} \lesssim \lvert  Q\rvert ^{1/p_j} \,.  
\end{align*}
 The proof in case
of $b^j$-functions is the same.
\end{proof}
%%%%%%%%%%%%%%%%%%%%%%%%%%%%%% PROOF PROOF PROOF

We need a variant of the $q$-universal inequality for the half-twisted differences to control several error terms that arise.
For $P\in\mathcal{G}^j$, let us define 
\begin{align}\label{e.Box}
	\Box _P^{\beta^j} f_j := \lvert  \widetilde D   _P ^{\beta ^{j}} f_j\rvert +  \widetilde \chi  _P \,,
\end{align}
where 
\begin{align}
	\label{e.sigmaQ}
	\widetilde \chi  _P := 
	\begin{cases}
	\mathbf 1_{P}  & \textup{a child of $ P$ is in $ \mathcal S ^{j}$}
	\\
	0 & \textup{otherwise}
 	\end{cases}
	\end{align}
The functions $\Box _P^{b^j} f_j$ are defined analogously.
If the applied function $\beta^j$ or $b^j$ is clear from the context, we omit the superscripts.
Now, the following $q$-universal inequality is a 
 consequence of sparsness of the stopping cubes $ \mathcal S ^{j}$ and 
the half-twisted inequality, Theorem \ref{t.mt},
\begin{equation}\label{e.Box<}
\Bigl\lVert 
\Bigl[
\sum_{\substack{P\in\mathcal{G}^j \;:\;  P\subset Q }}   
\lvert \Box  _{P}^{\beta^j} f_j \rvert ^2 
\Bigr] ^{1/2} 
\Bigr\rVert_{q}
\lesssim \lvert  Q\rvert ^{1/q}\,, \qquad 1< q < \infty  \,. 
\end{equation}
Here $Q$ is any cube in $\mathbf{R}^n$, and the corresponding inequality
is also true if we use $b^j$-functions.
For further details concerning the proof of \eqref{e.Box<}, we refer to \cite{lv-perfect}*{Section 5}.

%%%%%%%%%%%%%%%%%%%%%%%%%%%%%% SUBSECTION SUBSECTION SUBSECTION SUBSECTION
 %%%%%%%%%%%%%%%%%%%%%%%%%%%%%% SUBSECTION SUBSECTION SUBSECTION SUBSECTION 
 \subsection{An Estimate for Perturbations of $ b$}%\label{ss.}

 For a later discussion of the diagonal term in \S \ref{s.diagonal}, we
 need novel {\em perturbation inequalities} for the
twisted martingale differences.  
 We will first formulate and prove general statements, and only afterwards specialize to
 our setting.

 Let $ S_0$ be a dyadic cube, and $ \mathcal T$ be a collection of disjoint dyadic subcubes of $ S_0$. Let
$ \mathcal Q$ be the collection of all dyadic subcubes of $ S_0$ which are not contained in any $ T\in \mathcal T$. We suppose
 there are two admissible collections of functions\footnote{With exponent
 $p$ and constants $\sigma,\mathbf{B}$}: $b$ and $\beta$, and the corresponding
 terminal functions $b_T$ and $\beta_T$ for each $T\in\mathcal{T}$;
 for the definitions, we refer to \S \ref{ss.martingale}.
  Assume further that for all $ Q \in \mathcal Q$
  and $T\in\mathcal{T}$ there holds, for a fixed $ 0< \upsilon  < 8^{-1}$, 
 \begin{equation}\label{e:close}
 	 \int _{Q} \lvert b- \beta  \rvert ^{p} \; dx \le \upsilon^p \lvert  Q\rvert\,,\qquad
	 \int_T \lvert b_T-\beta_T\rvert^p \le \upsilon^p \lvert T\rvert\,.
 	 \end{equation}
These conditions say that $ b$ and $ \beta $, and
the corresponding terminal functions, are `close'.   
 
%%%%%%%%%%%%%%%%%%%%%%%%%%%%%% LEMMA LEMMA LEMMA
\begin{theorem}\label{l.perturb} 
Suppose inequalities \eqref{e:close} hold for some
$0<\upsilon<8^{-1}$, and $f\in L^1_{\textup{loc}}$ satisfies
$\lvert\langle f\rangle_Q\rvert \le \lambda$ for every cube $Q\in\mathcal{Q}\cup \mathcal{T}$.
Then, 
we have the following `perturbation inequality' 
	\begin{equation} \label{e.perturb}
\Bigl\lVert 
\Bigl[
\sum_{Q\in \mathcal{Q}} 
\bigl\lvert \bigl\{ \Delta ^{\beta} _Q - \Delta ^{b}_Q \bigr\}  f\bigr\rvert^2
\Bigr]^{1/2}
\Bigr\rVert_{p} \lesssim \upsilon\cdot 
\big\{ \lVert f\cdot \mathbf{1}_{S_0} \rVert_p + 
\lambda \lvert  S_0\rvert ^{1/p}\big\}\,. 
\end{equation}
Here the exponent $p$ is the one associated with functions $b$ and $\beta$, and the implied constant depends upon $n,p,\sigma,\mathbf{B}$.
\end{theorem}
%%%%%%%%%%%%%%%%%%%%%%%%%%%%%% LEMMA LEMMA LEMMA

This section is devoted to the proof, which is a variant of known techniques \cite{1011.1747,lv-perfect}.  
The proof relies on the crucial martingale transform inequality. The main lemma follows.
%  Also, the following restricted
%maximal function plays a role. Define
%\begin{equation}\label{e.fS}
%	f _{S_0}(x) := \sup_{x\in Q\in\mathcal{Q}} \langle \lvert f\rvert \rangle_Q\,.
%\end{equation} 
%It is important to realize that $f_{S_0}$ is constant
%on terminal cubes, and $\lVert f_{S_0}\rVert_p \lesssim \lVert f\rVert_p$.
%After this preparation, we are ready for
%the main Lemma.

%%%%%%%%%%%%%%%%%%%%%%%%%%%%%% LEMMA LEMMA LEMMA
\begin{lemma}\label{l.perturbHalf}  
Suppose that $\upsilon$ and $f$ are as in Theorem \ref{l.perturb}.
	Then, we have the inequality
	\begin{align} \label{e.perturbHalf}
\Bigl\lVert 
\Bigl[
\sum_{Q\in \mathcal{Q}} 
\bigl\lvert \big\{ D_Q^{\beta} - D_Q^b\big\} f\bigr\rvert^2
\Bigr]^{1/2}
\Bigr\rVert_{p} \lesssim \upsilon\cdot 
\big\{ \lVert f\cdot \mathbf{1}_{S_0} \rVert_p + 
\lambda \lvert  S_0\rvert ^{1/p}\big\}\,.
\end{align}
\end{lemma}
%%%%%%%%%%%%%%%%%%%%%%%%%%%%%% LEMMA LEMMA LEMMA

%%%%%%%%%%%%%%%%%%%%%%%%%%%%%% PROOF PROOF PROOF
\begin{proof}
We begin with preparations.
Fix $Q\in \mathcal{Q}$ and $Q'\in\textup{ch}(Q)\setminus \mathcal{T}$. 
Set $ \widetilde \beta = b - \beta $, 
and 
write $\beta_{k,Q} := (\langle \widetilde \beta\rangle_Q/\langle \beta\rangle_Q)^k$.
Define $\beta_{k,Q'}$ analogously.
Observe that the following inequalities hold for every $k\ge 1$:
\begin{equation}\label{e.control}
\lvert\beta_{k,Q'}\rvert + \lvert\beta_{k,Q}\rvert \le 2\cdot (4\upsilon)^k,\qquad
|\beta_{k,Q'}-\beta_{k,Q}\rvert \le \lvert \beta_{1,Q'} - \beta_{1,Q}\rvert\cdot k \cdot (8\upsilon)^{k-1}\,.
\end{equation}
Indeed, these follow from inequalities \eqref{e.ra} and \eqref{e:close},
and the fact that $Q,Q'\in\mathcal{Q}$. For the latter
inequality above, one also applies the mean value theorem. 

Then we write
\begin{align}
	\frac {1} {\langle \beta  \rangle_Q} -\frac {1} {\langle b   \rangle_Q}   
	&= 	\frac {1} {\langle \beta  \rangle_Q}\Bigl\{
	1 - \frac {\langle \beta  \rangle_Q} {\langle \beta  \rangle_Q +\langle \widetilde \beta  \rangle_Q }
	\Bigr\}\
	  \label{e.suM}
		=\frac {1} {\langle \beta  \rangle_Q} 
		\sum_{k=1} ^{\infty } (-1) ^{k+1} \beta_{k,Q} \,.  
\end{align}
Using the same expansion with $Q$ replaced by $Q'$ yields
inequality
\[
\bigl\lvert \bigl\{ D_Q^{\beta}- D_Q^b \bigr\} f\,\bigr\rvert  \cdot  \mathbf{1}_{Q'}\le 
\sum_{k=1}^\infty \biggl\lvert
\frac{\langle f\rangle_{Q'}}{\langle \beta \rangle_{Q'}} \beta_{k,Q'} - 
\frac{\langle f\rangle_{Q}}{\langle \beta \rangle_{Q}} \beta_{k,Q} \biggr\rvert \cdot \mathbf{1}_{Q'}\,.
%\Big\{ \langle f\rangle_{Q'} \Big[ \frac{1}{\langle \beta\rangle_{Q'}} -
%\frac{1}{\langle b\rangle_{Q'}}\Big] - \langle f\rangle_{Q} \Big[ \frac{1}{\langle \beta\rangle_{Q}} -
%\frac{1}{\langle b\rangle_{Q}}\Big] \Big\} \cdot \mathbf{1}_{Q'}
\]
Then, for a fixed $k$, write the summand on the right hand side as
\begin{align*}
&\biggl\lvert \beta_{k,Q} D_Q^\beta f \cdot \mathbf{1}_{Q'}+\frac{\langle f\rangle_{Q'}}{\langle \beta \rangle_{Q'}} 
\big\{ \beta_{k,Q'} - \beta_{k,Q}\big\} \cdot \mathbf{1}_{Q'}\biggr\rvert\\
&\le 2(4\upsilon)^k \cdot\lvert D_Q^\beta f\rvert  \cdot \mathbf{1}_{Q'}
+ 4\lambda\cdot k  \cdot (8\upsilon)^{k-1} \cdot   \lvert \beta_{1,Q'} - \beta_{1,Q}\rvert \cdot 
\mathbf{1}_{Q'}\,.
\end{align*}
Here we used  assumptions and inequalities \eqref{e.control}.
Observe that $\lvert \beta_{1,Q'} - \beta_{1,Q}\rvert\cdot
\mathbf{1}_{Q'} = \lvert D^\beta_Q \widetilde \beta\rvert \cdot \mathbf{1}_{Q'}$.
By summing the series over $k$, and then summing 
resulting estimates over $Q'\in\textup{ch}(Q)\setminus\mathcal{T}$,
\[
\bigl\lvert \bigl\{ D_Q^{\beta}- D_Q^b \bigr\} f\,\bigr\rvert\lesssim 
\upsilon\lvert D_Q^\beta f\rvert + \lambda \lvert D^\beta_Q \widetilde \beta\rvert
=\upsilon\lvert D_Q^\beta(f\cdot \mathbf{1}_{S_0})\rvert + \lambda \lvert D^\beta_Q(\widetilde \beta\cdot \mathbf{1}_{S_0})\rvert\,.
\]
Inequality
\eqref{e.perturbHalf} follows by using \eqref{e:close} and
universal martingale transform inequalities with $q=p$, see Theorem \ref{t.mt}.
\end{proof}
%%%%%%%%%%%%%%%%%%%%%%%%%%%%%% PROOF PROOF PROOF

\begin{proof}[Proof of Theorem~\ref{l.perturb}]
For $Q\in\mathcal{Q}$, we write $\Delta_Q^\beta f - \Delta_Q^b f $ as 
\begin{equation}
\begin{split}
\big\{ D_Q^{\beta} f - D_Q^b f\big\}\cdot b 
+ D_Q^{\beta} f \cdot (\beta-b)
+\sum_{\substack{Q'\in\textup{ch}(Q)\cap \mathcal{T}}} 
\big\{F^1_{Q'} - F^2_{Q'} - F^3_{Q'}\,\big\}\cdot \mathbf{1}_{Q'}\,,
\end{split}
\end{equation}
where we have denoted
$F^1_{Q'} := \langle f\rangle_{Q'}\cdot \{\beta_{Q'}-b_{Q'}\}$,
\[
F^2_{Q'} := \langle f\rangle_Q\cdot \bigg\{\frac{1}{\langle \beta\rangle_Q} - \frac{1}{\langle b\rangle_Q}\bigg\}
\cdot b,\quad F^3_{Q'}:= \langle f\rangle_Q\cdot \frac{1}{\langle \beta\rangle_Q}
\big\{\beta-b\big\}\,.
\]
Having Lemma \ref{l.perturbHalf} and martingale difference
inequalities, we can proceed as in \cite{lv-perfect}*{Section 2}.
For the convenience of the reader, we briefly recall this argument here.
Let us consider the square function of $D_Q^{\beta} f \cdot (\beta-b)$ first; to this
end, we define
\[
Sf := \bigg[\sum_{Q\in\mathcal{Q}} \big \lvert D_Q^{\beta} (f\cdot \mathbf{1}_{S_0})\big\rvert^2\bigg]^{1/2}\,,
\]
and consider the events $ E _{t } := \{ \lvert  Sf \rvert \ge t  \}\subset S_0$, where  $ t >0$. 
%	Let $S_\mathcal{T}\subset S_0$ be the union of terminal cubes $T\in \mathcal{T}$.
%By construction of $\mathcal{T}$, and Lebesgue Differentiation Theorem,
%we have $|\beta(x)-b(x)|\le \upsilon$ for almost every $x\in S_0\setminus S_{\mathcal{T}}$.
%Observe that $Sf$ is constant on terminal cubes, and denote
%$\mathcal{T}_t:=\{T\in\mathcal{T}\,:\,|Sf|\ge t\textrm{ on }T\}$. Since
%$T^{(1)}\in \mathcal{Q}$ for each terminal cube $T$,
%\begin{align*}
%\int_{E_t\cap S_{\mathcal{T}}} |\beta-b|^p\,dx &=\sum_{T\in\mathcal{T}_t}
%|T| \bigg(\frac{1}{|T|} \int_T |\beta-b|^p\,dx\bigg) \\&\le 2^n\upsilon^p \sum_{T\in\mathcal{T}_t}|T|
%\le 2^n\upsilon^p |E_t\cap S_{\mathcal{T}}|.
%\end{align*}
It is important to realize
that we
	can compare Lebesgue measure estimates and estimates with respect to $ \lvert  \beta-b \rvert ^{p} \; dx  $. 
	Namely,
by inequality \eqref{e:close}, the Lebesgue differentiation theorem, and
the fact that $Sf$ is constant on terminal cubes $T\in\mathcal{T}$, we obtain:
$\int_{E_t} |\beta-b|^p\,dx \le 2^n\upsilon^p |E_t|$.
Therefore, by  the Lebesgue measure estimates in Theorem~\ref{t.mt},  
\begin{equation*}
	\int _{S_0} \lvert  Sf\rvert ^{p } \lvert  \beta-b\rvert ^{p} \; dx  
	= p\int _{ 0} ^{\infty } t ^{p-1} \int _{E _{t }} \lvert  \beta-b\rvert ^{p} \; dx  \; d t 
	\lesssim \upsilon^p\lVert f\cdot \mathbf{1}_{S_0} \rVert_p ^{p} \,. 
	\end{equation*}
The square function of $\big\{ D_Q^{\beta} f - D_Q^b f\big\}\cdot b$
is estimated analogously, using
Lemma \ref{l.perturbHalf}, which also
contributes the constant $\upsilon$. 
The remaining square
functions, associated with $F_{Q'}^i$, $i=1,2,3$, 
are estimated by using the fact  that $\mathcal{T}$ is a disjoint collection
and Lebesgue measure is doubling.
In case of $i=2$, we also use expansion \eqref{e.suM} and
first inequality  in \eqref{e.control}.
\end{proof}

We specialize the perturbation inequalities to our setting.

\begin{theorem}\label{t.perturb_special}
For $j\in \{1,2\}$ and $0<\upsilon_1<4^{-1-n}$, we have the following inequality
	\begin{equation} \label{e.diff}
\Bigl\lVert 
\Bigl[
\sum_{Q\in \mathcal{G}^j}
\bigl\lvert \bigl\{ \Delta ^{\beta^j} _Q - \Delta ^{b^j}_Q \bigr\}  f_j\bigr\rvert^2
\Bigr]^{1/2}
\Bigr\rVert_{p_j} \lesssim \upsilon_1 \lvert Q^0\rvert^{1/p_j}\,.
\end{equation}
\end{theorem}

\begin{proof}
Let $ \mathcal R_{0}=\mathcal{A}^j_\ast$, 
and inductively set $ \mathcal R_{k+1}$ to be the maximal cubes $S'\in \mathcal S ^{j}$ strictly contained in any 
$S\in \mathcal R_{k}$.
Since  $\mathcal{S}^j$ is sparse, we have
$\sum_{S\in\mathcal{R}^k} \lvert S\rvert \lesssim \tau^k |Q^0\rvert$ if $k\ge 0$.
By 
disjointness of each collection
$\mathcal{R}^k$, the left hand side of \eqref{e.diff} is bounded by
\begin{align*}
%&\Big\lVert\sum_{Q\in \mathcal{G}^j}
%\bigl\lvert \bigl\{ \Delta ^{\beta^j} _Q - \Delta ^{b^j}_Q \bigr\}  f_j\bigr\rvert^2
%\Bigr]^{1/2}
%\Bigr\rVert_{p_j}
%&\lesssim \sum_{k=0}^\infty \bigg[
%\sum_{S\in\mathcal{R}^k}
% \int_{\mathbf{R}^n} \int_{\Omega} \Big\lvert 
% \sum_{\substack{Q\in\mathcal{G}^j\\ \pi_{\mathcal{S}^j}Q=S}}
% \epsilon_Q\bigl\{ \Delta ^{\beta^j} _Q - \Delta ^{b^j}_Q \bigr\}  f_j\Bigr\rvert^{p_j}
%\,d\epsilon dx\bigg]^{1/p_j}\\
%\lesssim 
\sum_{k=0}^\infty  \bigg[
\sum_{S\in\mathcal{R}^k}
\biggl\lVert
\biggl[
 \sum_{\substack{Q\in\mathcal{G}^j\\ \pi_{\mathcal{S}^j}Q=S}}
 \bigl|\bigl\{ \Delta ^{\beta^j} _Q - \Delta ^{b^j}_Q \bigr\}  f_j\bigr|^2\biggr]^{1/2}
\biggr\rVert_{p_j}^{p_j}\biggr]^{1/p_j}\,.
\end{align*}
Fix $k\ge 0$ and $S_0=S\in\mathcal{R}^k$.
The basic reduction to a  square function involving
cubes $Q\in\mathcal{Q}$ and
differences
$\Delta_Q^\beta$ and $\Delta_Q^b$ is described in
Remark \ref{r.mt}. 

We will apply Theorem \ref{l.perturb}, and therefore
we need to verify
that its assumptions are satisfied. 
To this end,
we may of course assume that there is a cube $Q\in\mathcal{G}^j$ such
that $\pi_{\mathcal{S}^j}Q=S_0$. As a consequence,
$S_0\in\mathcal{S}^j$ is not contained in any cube in $\mathcal{B}^j$, and
a case study using definition of Type C cubes
shows that
$\lvert \langle f_j\rangle_Q\rvert \lesssim 1$ if
$Q\in\mathcal{Q}\cup \mathcal{T}$ and
$\lVert f_j\cdot \mathbf{1}_{S_0}\rVert_{p_j}\lesssim \lvert S_0\rvert^{1/p_j}$ 
(recall that $f_j$ is in dyadic $\mathrm{BMO}$
and the normalization by $\Lambda^{-1}$ takes place).
Moreover, the same fact about
$S_0$ combined with definition of Type A cubes
implies that 
\[
\int_Q \lvert b-\beta\rvert^p = \int_Q \lvert\widetilde\beta^j_{S_0}\rvert^{p_j} \le 2^n\upsilon_1^{p_j}\lvert Q\rvert\,,\qquad Q\in\mathcal{Q}\,.
\]
For the second condition in \eqref{e:close}
for $T\in\mathcal{T}$, we first
observe that $b_T-\beta_T=0$ if $T$ is a child of a cube in $\mathcal{B}^j$.
In complementary case, $T\in \textup{ch}_{\mathcal{S}^j}(S_0)$,
and its parent is not contained in any cube in $\mathcal{B}^j$.
Thus,
$\int_T \lvert b_T-\beta_T\rvert^p = \int_T \lvert \widetilde{\beta}^j_T\rvert^{p_j}
\le \upsilon_1^{p_j} \lvert T\rvert$
by definition of Type A cubes.

The proof is finished by using Theorem \ref{l.perturb} and
appealing to previous inequalities.
\end{proof}

%%%%%%%%%%%%%%%%%%%%%%%%%%%%%% SUBSECTION SUBSECTION SUBSECTION SUBSECTION
 %%%%%%%%%%%%%%%%%%%%%%%%%%%%%% SUBSECTION SUBSECTION SUBSECTION SUBSECTION 

%%%%%%%%%%%%%%%%%%%%%%%%%%%%%% SECTION  SECTION SECTION
%%%%%%%%%%%%%%%%%%%%%%%%%%%%%% SECTION  SECTION SECTION 
\section{The Inner Product and the Main Term} \label{s.innerProduct}

During the course of the remaining sections, 
we prove inequality \eqref{e.GS<}, namely,
\begin{equation}
	\Bigl\lvert 
 \sum_{P\in \mathcal G ^{1} }	\sum_{Q\in \mathcal G ^{2} } 
			\langle  T  \Delta _P ^{\beta^1 } f_1,  \Delta _Q ^{\beta^2 }f_2\rangle 
			\Bigr\rvert \le 
			 \bigl\{C_2
			\{ 1+ \mathbf T _{\textup{loc}}\}
			 + C_3   r\upsilon _1 \Lambda^2 \mathbf T  \bigr\} 
			 \lvert  Q ^{0}\rvert\,,
\end{equation}
where $C_3$ is a constant not
allowed to depend upon the absorption parameters.
This inequality
completes the proof of Lemma~\ref{l.sharp} which, in turn, implies our main result.
Let us recall that the functions $f_j$ have
been normalized, allowing us to assume that $\Lambda=1$.

The sum above is split into  dual triangular sums, one
of which is the sum over 
$ (P,Q) \in \mathcal G ^{1} \times \mathcal G ^{2}$ such that $ \ell P\ge \ell Q$.  By using goodness this
triangular sum is split into 
different collections: 
\begin{gather*}
		\mathcal P _{\textup{far}} := 
	\{ (P,Q) \in \mathcal G ^{1} \times \mathcal G ^{2}  
	\;:\; 3P \cap Q = \emptyset\,, \ell Q \le \ell P \}\,;
	\\
	\mathcal P _{\textup{diagonal}} := 
	\{ (P,Q) \in \mathcal G ^{1} \times \mathcal G ^{2} \setminus \mathcal P _{\textup{far}} \;:\; 2 ^{-r} \ell P \le \ell Q \le \ell P \}\,;
	\\
	\mathcal P _{\textup{nearby}} := 
	\{ (P,Q) \in \mathcal G ^{1} \times \mathcal G ^{2} \setminus \mathcal P _{\textup{diagonal}} 
	\;:\; Q \subset 3P\setminus P\}\,;
	\\
	\mathcal P _{\textup{inside}} := 
	\{ (P,Q) \in \mathcal G ^{1} \times \mathcal G ^{2} \setminus \mathcal P _{\textup{diagonal}} 
	\;:\;  Q \subset P\}\,.
\end{gather*}
The sums over these collections  are handled separately and, aside from the `inside' 
and `diagonal' terms, one can sum over
the absolute value of the inner products.  
The main tools to control these terms include the twisted martingale transform inequalities
combined with the local $ Tb$ hypothesis.
All of the cubes are good, which is a point used systematically. 
	This useful fact is  frequently combined with the smoothness condition on the kernel, 
	to conclude that certain maximal functions applied to the $ \beta $ functions appear.  That these maximal functions are controlled 
	will be a consequence of the corona construction, combined with the universal half-twisted martingale inequalities. 
In the analysis of the diagonal term, the perturbation
inequalities established in \S\ref{s.inequalities} play a key role.

In this section, we concentrate on the 
`inside' term, which is the main term.    
The conditions for $ (P,Q) \in	\mathcal P _{\textup{inside}} $ are: $ Q\subset P$,
 $ 2 ^{r} \ell Q < \ell P$, and $(P,Q)\in\mathcal{G}^1\times\mathcal{G}^2$; 
these conditions are abbreviated  $ Q\Subset P$ below.  
Even though $ Q$ is in a different grid from that of $P$, a child of $ P$ contains $Q$ because of goodness, and we denote that 
child by $ P_Q$.  
 We will write $\Delta_P:=\Delta_P^{\beta^1}$ (likewise for $Q$) and
$
	\widetilde \Delta _P f_1:= \widetilde D_P f_1\cdot \beta ^{1} _{\pi _{\mathcal S^1} P} 
$, 
where the half-twisted martingale difference $\widetilde D_P=\widetilde D_P^{\beta^1}$ of \eqref{e.halfDef} does not sum over of the children of $ P$ that 
have a different stopping parent from that of $P$. 

In order to control the inside term, it suffices to bound the sum over $ S \in \mathcal S ^{1}$ of the terms 
\begin{equation}\label{e.decomposition}
	\bigg\lvert\mathbf{1}_{\{\pi S\in \mathcal{G}^1\}}\cdot \sum_{\substack{Q \;:\; Q\Subset \pi S}} \langle f_1 \rangle _{S} 
	\langle T \beta ^{1} _{S}, \Delta _Q f_2\rangle\bigg\rvert + 
	\bigg\lvert \underbrace{\sum_{P \;:\; \pi _{\mathcal S^1} P=S}
	\sum_{\substack{Q \;:\; Q\Subset P}} 
	\langle  T   \widetilde \Delta _P f_1,  \Delta _Q f_2\rangle}_{=:B_S(f_1,f_2)}\bigg\rvert \,. 
	\end{equation}
The point of this step is that, in the left hand side, the argument of $T$ depends only on $ \beta ^{1}_S$.  And, a sufficient cube-wise inequality is 
\begin{equation*}
	\eqref{e.decomposition} \lesssim \widetilde{\mathbf{T}}_{\textup{loc}} \lvert S\rvert\,,\qquad 
	\widetilde{\mathbf{T}}_{\textup{loc}}:=\mathbf T _{\textup{loc}} + \upsilon _1 \mathbf T\,,
\end{equation*}
where the implied constant
is not allowed to depend upon the absorption parameters.
%The term $\widetilde{\mathbf{T}}_{\textup{loc}}$  appears in Lemma~\ref{l.typeA}, condition (3c).  
Since the collection $ \mathcal S ^{1}$ is 
sparse, this upper bound is  summable over $ S\in \mathcal S ^{1} $ to a multiple of
$\widetilde{\mathbf{T}}_{\textup{loc}} \lvert  Q ^{0}\rvert $.  

The left-hand side of \eqref{e.decomposition} is easy to control. First of all, by the local $ Tb$ properties stated in
Lemma \ref{l.typeA}, and the  twisted martingale inequality \eqref{e.twisted}, 
\begin{align*}
\Bigl\lvert \mathbf{1}_{\{\pi S\in \mathcal{G}^1\}}\cdot
	\sum_{\substack{Q \;:\; Q\Subset \pi S \\ Q\subset S}} \langle f_1 \rangle _{S} 
	\langle T \beta ^{1} _{S}, \Delta _Q f_2\rangle 
	\Bigr\rvert & \lesssim \widetilde{\mathbf T} _{\textup{loc}}
	\lvert  S\rvert ^{1/p_2'} 
	\Bigl\lVert 
		\sum_{\substack{Q \;:\; Q\Subset \pi S}} \langle f_1 \rangle _{S}  \Delta _Q f_2
		\Bigr\rVert_{p_2} \lesssim  \widetilde{\mathbf T}_{\textup{loc}}  \lvert  S\rvert \,. 
\end{align*}
The remaining part of the left-hand side is a sum over
cubes $Q\Subset \pi S$ for which $Q\cap S=\emptyset$.
This part is conveniently estimated by using Hardy's inequality in Lemma \ref{l.hardy}
and inequality $p_2'\le p_1$.

In the right-hand side of \eqref{e.decomposition} the argument of $ T$ is written as follows. If $ Q\Subset P$ and $\pi _{\mathcal S^1} P=S$,
\begin{align*}
	\widetilde \Delta _P f_1  &= 
	\langle  \widetilde D_P f_1 \rangle _{P_Q}  \cdot \beta ^{1} _{ S } \mathbf 1_{S} 
	- \langle  \widetilde D_P f_1 \rangle _{P_Q} \cdot \beta ^{1} _{S} 
	\mathbf 1_{S\setminus P_Q} 
	+ \widetilde \Delta _P f_1 \cdot \mathbf 1_{P\setminus P_Q} \\
	&=: 
	\Delta ^{\textup{para}} _{P} f_1 - 	\Delta ^{\textup{stop}} _{P} f_1 + 	\Delta ^{\textup{error}} _{P} f_1\,,
	\end{align*}
where we treat $ \widetilde \Delta _P f_1  \mathbf 1_{P_Q}$ as the main contribution, and write 
$ \mathbf 1_{P_Q}= \mathbf 1_{S} - \mathbf 1_{S \setminus P_Q}$.  
This decomposition of $\widetilde \Delta _P f_1$ leads to a corresponding decomposition of $ B _{S} (f_1,f_2)$---by which we denote 
 the second term on display \eqref{e.decomposition}
 without the absolute values---into the paraproduct term, the stopping term, and the error term, written as 
\begin{align*}
	B _{S} (f_1,f_2) = B _{S} ^{\textup{para}} (f_1,f_2)  
	- B _{S} ^{\textup{stop}} (f_1,f_2) + B _{S} ^{\textup{error}} (f_1,f_2)\,,
	\end{align*}
where $S\in\mathcal{S}^1$ is fixed. The terminology is drawn from \cite{V,1003.1596}. 

%%%%%%%%%%%%%%%%%%%%%%%%%%%%%% SUBSECTION SUBSECTION SUBSECTION SUBSECTION
 %%%%%%%%%%%%%%%%%%%%%%%%%%%%%% SUBSECTION SUBSECTION SUBSECTION SUBSECTION 
 \subsection{Control of the Paraproduct Term}%\label{ss.}
 This brief argument is in fact the core of the proof.  
 Consider $ B ^{\textup{para}} _{S}(f_1,f_2)$.  
 In this term, the argument of $ T$ is a certain multiple of $ \beta ^{1} _{S}\mathbf{1}_S=\beta ^{1} _{S}$. For the  cubes $ Q\in\mathcal{G}^2$, 
 let us define 
\begin{equation*}
	\varepsilon _Q := \sum_{\substack{{P \;:\; \pi _{\mathcal S^1} P=S}\\ Q \Subset P}}  \langle  \widetilde D_P f_1\rangle _{P_Q}\,.
\end{equation*} 
The condition\footnote{The condition applies with  
the minimal cube in $\mathcal{G}^1$,
subject to the summation conditions, instead of $Q$.}   \eqref{e.BNDD} of Lemma~\ref{l.GB},
 was designed for the implication that 
the  numbers $ \varepsilon _Q$ are uniformly bounded.    Therefore, we can estimate
\begin{align*}
\bigl\lvert 
	B _{S} ^{\textup{para}} (f_1,f_2)  
\bigr\rvert & = 
\Bigl\lvert 
 \sum_{P \;:\; \pi _{\mathcal S^1} P=S} \sum_{Q \;:\; Q\Subset P} 
\langle  \widetilde D_P f_1\rangle _{P_Q}   \cdot 	\langle   T \beta ^{1} _{S}  ,  \Delta _Q f_2\rangle 
\Bigr\rvert
\\
& = 
\Bigl\lvert \Bigl\langle T \beta ^{1} _{S} ,
\sum_{Q \;:\; Q\Subset S} \varepsilon _Q \Delta _Q f_2  \Bigr\rangle_{}
\Bigr\rvert
 \le \lVert \mathbf{1}_S\cdot T \beta ^{1} _{S} \rVert_{p_2'}  
\Bigl\lVert 
\sum_{Q \;:\; Q\Subset S} \varepsilon _Q \Delta _Q f_2 
\Bigr\rVert_{p_2} 
\lesssim 
 \widetilde {\mathbf T } _{\textup{loc}}   \lvert  S\rvert \,,  
\end{align*}
where we appealed to the local $ Tb$ hypothesis, condition (3c) of Lemma \ref{l.typeA}, and  the 
martingale transform inequality \eqref{e.twisted}. 
This completes the analysis of the paraproduct term.

%%%%%%%%%%%%%%%%%%%%%%%%%%%%%% SUBSECTION SUBSECTION SUBSECTION SUBSECTION
 %%%%%%%%%%%%%%%%%%%%%%%%%%%%%% SUBSECTION SUBSECTION SUBSECTION SUBSECTION 
 \subsection{The Stopping Term}%\label{ss.}
 Recall that 
 \begin{equation*}
 	 \lvert B _{S} ^{\textup{stop}} (f_1,f_2)  \rvert
  = 
\bigg\lvert \sum_{P \;:\; \pi _{\mathcal S^1} P=S} \sum_{Q \;:\; Q\Subset P} 
 \langle  \widetilde D_P f_1\rangle _{P_Q} 	\langle   T (\beta ^{1} _{S}  \mathbf 1_{S \setminus P_Q}),  \Delta _Q f_2\rangle\bigg\rvert\,.
\end{equation*}
We will bound this by a constant multiple of $ \lvert  S\rvert $ via appealing
to that (a)  $ \int \Delta _Q f_2 = 0$ and the kernel of $ T$ has smoothness, and that
(b)  the universal  half-twisted inequality 
\eqref{e.Box<} is valid.
% We will use this  argument  in those cases when the cubes are relatively close.  

For integers $ s > r$, we restrict the side length of $ Q$ so that $ 2 ^{s} \ell Q= \ell P$, and 
thereby obtain a 
geometric decay in $ s$.  To accommodate this, let us define 
\begin{equation*}
	B _{S, s} ^{\textup{stop}} (f_1,f_2)  
 := 
\sum_{P \;:\; \pi _{\mathcal S^1} P=S} \sum_{\substack{Q \;:\; Q\Subset P\\  2 ^{s} \ell Q= \ell P }} 
\langle  \widetilde D_P f_1\rangle _{P_Q} 	\langle   T (\beta ^{1} _{S}  \mathbf 1_{S \setminus P_Q}),  \Delta _Q f_2\rangle 
\end{equation*}
By goodness, $ \textup{dist} (S \setminus P_Q,Q) \ge  (\ell Q) ^{\epsilon } (\ell P_Q) ^{1- \epsilon }$. Therefore, by  
the smoothness condition on the kernel and the mean zero property of 
$ \Delta _Q f_2$  we can estimate the inner product as follows; let  
$ x_Q$ be the center of $ Q$ and recall also definition \eqref{e.Box}.
\begin{align*}
	\lvert 	\langle   T (\beta ^{1} _{S}  \mathbf 1_{S \setminus P_Q}),  \Delta _Q f_2\rangle \rvert
	& = 
	\Bigl\lvert 	 \int _{Q} \int _{S\setminus P_Q} 
\{ K (x,y) - K (x_Q,y)\} \beta ^{1} _{S} (y) \Delta _Q f_2 (x) \; dy dx  
\Bigr\rvert
\\
& \lesssim 
	 \int _{Q} \int _{S\setminus  P_Q}
	\frac { (\ell Q) ^{\eta }} { \lvert  x-y\rvert ^{n+ \eta } } \lvert \beta ^{1} _{S} (y) \Delta _Q f_2 (x) \rvert \; dy dx
	\\
& \lesssim 
2 ^{- \eta ' s} \inf _{x\in Q} M \beta ^{1} _{S} (x)\cdot \int_Q    \Box_Q f_2 \; dx\,.
\end{align*}
This is a standard off-diagonal estimate, by splitting
the region of integration in appropriate annuli, combined with the goodness of $ Q$ and the 
properties of our corona construction.  
Observe that we  gained a geometric decay in $ s$ with
$ \eta ' = (1-\epsilon)\cdot \eta >0$.  

Since cubes $Q$ with same side length, specified by $P$, are disjoint,
there is a simple appeal to the Cauchy--Schwarz inequality.
Following that, we use the trilinear form of H\"older's inequality, with indices $ p_1, 2 p_1', 2p_1'$, and the universal 
half-twisted  inequality \eqref{e.Box<}. By doing so, we obtain
\begin{align}  \label{e.again}
	\lvert B _{S, s} ^{\textup{stop}} (f_1,f_2) \rvert    & 
	\lesssim 2 ^{- \eta ' s}
	\sum_{P \;:\; \pi _{\mathcal S^1} P=S} \sum_{\substack{Q \;:\; Q\Subset P\\  2 ^{s} \ell Q= \ell P }}  
	\langle \lvert  \widetilde D_P f_1\rvert \rangle _{P}   \int_Q   M \beta ^{1} _{S} \cdot \Box_Q f _2  \; dx 
	\\
	& \lesssim 2 ^{- \eta ' s}
	\int _{S}  M \beta ^{1} _{S}  
	\Bigl[	\sum_{P \;:\; \pi _{\mathcal S^1} P=S}  \langle \lvert   \widetilde D_P f_1\rvert 
	\rangle _{P} ^2 \cdot \mathbf 1_{P}  \Bigr] ^{1/2} 
	\Bigl[ 
	\sum_{\substack{Q \;:\; Q\Subset S}} 
	\lvert  \Box_Q f_2\rvert ^2    \Bigr] ^{1/2} 
	\; dx 
	\\
	& \lesssim 2 ^{- \eta ' s} \lvert  S\rvert ^{1/p_1} 
	\Bigl\lVert 
	\Bigl[	\sum_{P \;:\; \pi _{\mathcal S^1} P=S}  \lvert  M \widetilde D_P f_1\rvert ^2  \Bigr] ^{1/2} 
	\Bigr\rVert_{2 p_1'} 
	\Bigl\lVert \Bigl[ 
	\sum_{\substack{Q \;:\; Q\Subset S}} 
	 \lvert \Box_Q f_2\rvert ^2   \Bigr] ^{1/2} 
	 \Bigr\rVert_{2 p_1'} 
	 \\&\lesssim 2 ^{- \eta ' s} \lvert S\rvert \,. 
\end{align}
This completes the analysis of stopping term.

%%%%%%%%%%%%%%%%%%%%%%%%%%%%%% SUBSECTION SUBSECTION SUBSECTION SUBSECTION
 %%%%%%%%%%%%%%%%%%%%%%%%%%%%%% SUBSECTION SUBSECTION SUBSECTION SUBSECTION 
 \subsection{The Error Term}%\label{ss.}
Here we need to control 
 \begin{align*}
 	\lvert B ^{\textup{error}} _{S} (f_1,f_2)\rvert  = \bigg\lvert\sum_{s =r +1} ^{\infty }
 	 \sum_{P \;:\; \pi _{\mathcal S^1} P=S} \sum_{\substack{Q \;:\; Q\Subset P\\ 2 ^{s} \ell Q= \ell P }} 
 	  \langle   T (\widetilde  \Delta _{P} f_1 \cdot \mathbf 1_{P \setminus P_Q})  ,  \Delta _Q f_2\rangle
	  \bigg\rvert \,. 
 	  \end{align*}
  For a fixed $ s> r$, we call the inner double series above $  B ^{\textup{error}} _{S, s} (f_1,f_2) $. 
  We will obtain a geometric 
  decay in $ s$, by using essentially the same argument as in the treatment of stopping term. 

  Indeed, 
  \begin{align*}
  	  \lvert 	\langle   T (\widetilde  \Delta _P f_1\cdot \mathbf 1_{P \setminus  P_Q}),  \Delta _Q f_2   \rvert
	& = 
	\Bigl\lvert 	\int _{Q} \int _{P \setminus P_Q} 
	\{ K (x,y) - K (x_Q,y)\}  \widetilde \Delta _P f_1 (y) \Delta _Q f_2 (x) \; 
	dy dx  \Bigr\rvert
\\
& \lesssim 
	 \int _{Q} \int _{P\setminus P_Q}
	\frac { (\ell Q) ^{\eta }} { \lvert  x-y\rvert ^{n+ \eta } } \lvert \widetilde \Delta _P f_1(y) \Delta _Q f_2 (x) \rvert \; dy dx 
	\\
& \lesssim 
2 ^{- \eta ' s} \cdot
\langle  \lvert  \widetilde D_Pf_1\rvert  \rangle_P \cdot
\inf _{x\in Q} M \beta ^{1} _{S} \cdot \int_Q \Box _Q f_2 \; dx \,.
\end{align*} 
Repeating the inequalities starting from \eqref{e.again} gives
$
	\lvert B ^{\textup{error}} _{S, s} (f_1,f_2)  \rvert  \lesssim 2 ^{- \eta ' s}\lvert  S\rvert 
$, and this suffices for the error term.

\section{The Remaining Terms} \label{s.remaining}

In this section we estimate all the remaining terms `nearby', `far', and `diagonal'.
%%%%%%%%%%%%%%%%%%%%%%%%%%%%%% SUBSECTION SUBSECTION SUBSECTION SUBSECTION
 %%%%%%%%%%%%%%%%%%%%%%%%%%%%%% SUBSECTION SUBSECTION SUBSECTION SUBSECTION 
 \subsection{The Nearby Term}\label{s.nearby}
 The nearby term concerns pairs of cubes $(P,Q)\in\mathcal{P}_{\textup{nearby}}$, that is,
cubes 
in $\mathcal{G}^1\times \mathcal{G}^2$
with the properties $ 2 ^{r} \ell Q< \ell P$ and $ Q\subset 3P\backslash P$.  
This term can be written as a sum over $S\in\mathcal{S}^1$ of terms
 \begin{equation}\label{e.nearby_decomposition}
\mathbf{1}_{\{\pi S\in \mathcal{G}^1\}}\sum_{\substack{Q \;:\; Q\subset 3\pi S \setminus \pi S\\ 2^{r} \ell Q < \ell \pi S }} 
\langle f_1\rangle_S \cdot \langle T \beta_S^1,  \Delta _Qf_2\rangle
+  \sum_{P:\pi_{\mathcal{S}^1}P=S} \sum_{\substack{Q:Q\subset 3P\setminus P\\ 2^r\ell Q < \ell P}}
\langle T\widetilde \Delta_P f_1, \Delta_Q f_2\rangle\,,
\end{equation}
where we tacitly assume that $P\in \mathcal{G}^1$ and $Q\in  \mathcal{G}^2$.
For a fixed $S\in\mathcal{S}^1$, the absolute value of the double series above is estimated by 
\begin{equation}\label{e.close}
\sum_{s>r}\sum_{P:\pi_{\mathcal{S}^1}P=S} \sum_{\substack{Q:Q\subset 3P\setminus P\\ 2^s\ell Q = \ell P}}
\lvert\langle T\widetilde \Delta_P f_1, \Delta_Q f_2\rangle\rvert\,.
\end{equation}
By using Lemma \ref{l.off} below and  following the arguments in \eqref{e.again} with obvious changes, we find
that the inner double series in \eqref{e.close}, with  a fixed $s>r$,  is dominated by
\begin{align*}
2^{-s\eta'}\sum_{P:\pi_{\mathcal{S}^1}P=S} \sum_{\substack{Q \;:\; Q\subset 3P \backslash P\\ 2 ^{s} \ell Q= \ell P }}  
 \inf _{x\in Q} M \beta ^{1} _{S} (x)\cdot
 \langle  \lvert \widetilde D_P f_1\rvert \rangle _P\cdot \int _{Q} \Box _Q f_2(x) \; dx 
 \lesssim 2^{-s\eta'} \lvert  S\rvert\,. 
\end{align*}
The right hand side is summable in $s$ to a constant multiple of $\lvert S\rvert$.
Consequently, by  applying the sparseness of $\mathcal{S}^1$, we find
that
\eqref{e.close} summed over $S\in\mathcal{S}^1$ is bounded by a constant multiple of $\lvert Q^0\rvert$.
The same method of proof controls the first term in \eqref{e.nearby_decomposition}; 
alternatively, one may apply the Hardy's inequality, Lemma \ref{l.hardy}. 

We now turn to a lemma that is used above.

%We conclude our analysis of the nearby term by 
%the basic Lemma that relies on goodness of $Q$.

%%%%%%%%%%%%%%%   LEMMA LEMMA LEMMA
\begin{lemma}\label{l.off} Let 
 $ (P, Q)\in\mathcal{P}_{\textup{nearby}} $ with $\pi_{\mathcal{S}^1}P=S$. Then
 with $\eta'=\eta(1-\epsilon)>0$ we have
	\begin{align}\label{e.nearby} 
 	 \lvert  \langle  T \widetilde\Delta _{P} f_1, \Delta _Q f_2\rangle\rvert 
 	 &\lesssim   \bigl(\ell Q/\ell P \bigr) ^{\eta '}  
 	 \cdot \inf _{x \in Q} M \beta ^{1}_S (x)\cdot  \langle \lvert \widetilde D_P f_1 \rvert \rangle_P \cdot
 	 \int _{Q} \Box _Q f_2  \; dx \,. 
	 \end{align}
\end{lemma}
%%%%%%%%%%%%%%%%% LEMMA LEMMA LEMMA

%%%%%%%%%%%%%% PROOF PROOF PROOF
\begin{proof}
By assumption, $Q\subset 3P\setminus P$ and $2^r \ell Q<\ell P$.
Since $Q$ is good, 
$|x-x_Q| \le |y-x_Q|/2$ for every $x\in Q$ and $y\in P$. Hence,
the kernel smoothness condition applies, and
we can estimate as follows, with  $ x_Q$ the center of $ Q$,
	\begin{align*}
 \lvert  \langle  T \widetilde\Delta _{P} f_1, \Delta _Q f_2\rangle\rvert  
 &= \Bigl\lvert 
 \int _{Q} \int _{P} \{ K (x,y) - K (x_Q,y)\} \widetilde\Delta _P f_1(y) \Delta _Q f_2 (x) \; dy dx 
 \Bigr\rvert
\\
& \lesssim 
  \int _{Q} \int _{P} 
  \frac {|x-x_Q|^\eta} {\lvert  x_Q-y\rvert^{n+ \eta }}
  \bigl\lvert  \widetilde\Delta _P f_1(y) \Delta _Q f_2 (x)\bigr\rvert \; dy dx 
  \\
 & \lesssim  
 \bigl( \ell Q / \ell P\bigr)^{\eta(1-\epsilon)} \cdot
  \inf _{x \in Q} M \widetilde \Delta _{P} f_1  (x) \cdot\int _{Q} \lvert  \Delta _Q f_2\rvert\; dx\,.
\end{align*}
Since $Q\in\mathcal{G}^2$, we have
$	\int _{Q}\lvert  \Delta _Q f_2\rvert\; dx  \lesssim \int _{Q}  \Box _Q f_2  dx$.
Furthermore, by definition,
\begin{align*}
\lvert \widetilde\Delta_P f_1 \rvert &= \lvert \widetilde D_P f_1 \cdot \beta_S^1\rvert
= \sum_{P'\in\textup{ch}(P)} |\langle \widetilde D_P f_1\rangle_{P'}\rvert \cdot |\beta_S^1 \cdot \mathbf{1}_{P'}\rvert \lesssim \langle \lvert \widetilde D_P f_1\rvert \rangle_P \cdot |\beta_S^1|\,.
\end{align*}
In particular, $M\widetilde \Delta_P f_1 \lesssim \langle \lvert \widetilde D_P f_1|\rangle_P\cdot M\beta_S^1$,
so the desired estimate follows.
\end{proof}
%%%%%%%%%%%%  PROOF PROOF PROOF

%%%%%%%%%%%%%%%%%%%%%%%%%%%%%% SUBSECTION SUBSECTION SUBSECTION SUBSECTION
 %%%%%%%%%%%%%%%%%%%%%%%%%%%%%% SUBSECTION SUBSECTION SUBSECTION SUBSECTION 
 \subsection{The Far Term}\label{s.farTerm}
 The far term concerns pairs of cubes  $(P,Q)\in\mathcal{P}_{\textup{far}}$,
satisfying $\ell Q\le \ell P$ and 
 $3P\cap Q = \emptyset $ in particular.     The goodness of these cubes is irrelevant here.
The absolute value of the far term is  bounded 
by the sum over integers $ s\ge 0$ and $ t\ge 1$ of  terms
  \begin{equation}\label{e.what_to_bound}
 \sum_{P} \sum_{\substack{Q \;:\;  2 ^{t-1} \ell P\le \textup{dist} (P,Q) < 2 ^{t} \ell P \\ 2 ^{s} \ell Q= \ell P }} 
\mathbf{1}_{(P,Q)\in\mathcal{P}_{\textup{far}}} \cdot
\lvert  	 \langle T \Delta _Pf_1,  \Delta _Qf_2\rangle\rvert\,.
\end{equation}
By
Lemma \ref{l.basic_far} below, we obtain the following upper bounds for the term \eqref{e.what_to_bound}
\begin{equation}\label{e.last}
\begin{split}
&2 ^{- \eta s -(n+ \eta)t} 
\int _{\mathbf{R}^n}
\sum_{P} \sum_{\substack{Q \;:\;  2 ^{t-1} \ell P\le \textup{dist} (P,Q) < 2 ^{t} \ell P \\ 2 ^{s} \ell Q= \ell P }} 
  \langle \Box_P f_1\rangle_P \cdot  \mathbf{1}_Q(x)\cdot
 	  	  \Box_Q f_2(x) \; dx 
\\
& \lesssim 2 ^{- \eta s -\eta t}  
\Bigl\lVert 
\Bigl[
\sum_{P\in\mathcal{G}^1} \lvert M\Box_P f_1\rvert^2
\Bigr] ^{1/2} 
\Bigr\rVert_{2}
\Bigl\lVert 
\Bigl[
\sum_{Q\in\mathcal{G}^2}  \lvert \Box_Q f_2  \rvert^2
\Bigr] ^{1/2} 
\Bigr\rVert_{2}
\\
& \lesssim 2 ^{- \eta (s+t)} \lvert  Q ^{0}\rvert\,.  
 	 \end{split}
	 \end{equation}
	 Observe that in the first estimate
	 we lose a factor $2^{nt/2}$ twice, 
	 because of additional summation 
	 associated with both of the square functions. 
	 In order to see this for the first square function, one changes the
	 order of summation and integration, and then applies inequality $|P|^{-1}\sum_{Q} |Q|\lesssim 2^{tn}$
	 for each $P$ inside the $P$-summation.
 
The last bound in \eqref{e.last} is still summable in $ s$ and $ t$, so that
we are left with the following.
  
 \begin{lemma}\label{l.basic_far}
 Let $(P,Q)\in \mathcal P _{\textup{far}}$.
 Then
 \begin{align*}
 \lvert \langle T\Delta_P f_1, \Delta_Q f_2\rangle \rvert
 \lesssim \bigl(\ell Q/\ell P \bigr)^{\eta} \cdot \bigg(\frac{\mathrm{dist}(P,Q)}{\ell P}\bigg)^{-n-\eta}\,
 \cdot \langle \Box_P f_1\rangle_P \cdot \int_Q \Box_Q f_2\,.
 \end{align*}
 \end{lemma}
 
 \begin{proof}
 Since $\mathrm{dist}(P,Q)\ge \ell P$, 
 the kernel smoothness condition applies with  $ x_Q$  the center of $ Q$:
	\begin{align*}
 \lvert  \langle  T \Delta _{P} f_1, \Delta _Q f_2\rangle\rvert  
 &= \Bigl\lvert 
 \int _{Q} \int _{P} \{ K (x,y) - K (x_Q,y)\} \Delta _P f_1(y) \Delta _Q f_2 (x) \; dy dx 
 \Bigr\rvert
\\
& \lesssim 
  \int _{Q} \int _{P} 
  \frac {|x-x_Q|^\eta} {\lvert  x_Q-y\rvert^{n+ \eta }}
  \bigl\lvert  \Delta _P f_1(y) \Delta _Q f_2 (x)\bigr\rvert \; dy dx 
  \\
 & \lesssim \bigl(\ell Q\bigr)^{\eta} \cdot \mathrm{dist}(P,Q)^{-\eta} \cdot
  \inf _{x \in Q} M \Delta _{P} f_1 (x) \cdot \int _{Q} \lvert  \Delta _Q f_2\rvert\; dx\,.
\end{align*}
Observe that
\begin{equation*}
	\int _{P}\lvert  \Delta_P f_1\rvert\; dx  \lesssim \int _{P}  \Box_P f_1\; dx,\qquad 
		\int _{Q}\lvert  \Delta _Q f_2\rvert\; dx  \lesssim \int _{Q}  \Box _Q f_2  dx \,.  
\end{equation*}
Thus,
\begin{align*}
M \Delta_P  f_1(x_Q) \lesssim \frac{1}{\mathrm{dist}(P,Q)^n} \int_{\mathbf{R}^n} \lvert \Delta_P f_1\rvert\; dx
\lesssim \frac{|P|}{\mathrm{dist}(P,Q)^n} \langle \Box_P f_1\rangle_P\,.
\end{align*}
The desired estimate follows by combining the estimates above.
 \end{proof}
%%%%%%%%%%%%%%%%%%%%%%%%%%%%%% SECTION  SECTION SECTION
%%%%%%%%%%%%%%%%%%%%%%%%%%%%%% SECTION  SECTION SECTION 

%%%%%%%%%%%%%%%%%%%%%%%%%%%%%% SECTION  SECTION SECTION
%%%%%%%%%%%%%%%%%%%%%%%%%%%%%% SECTION  SECTION SECTION 
\subsection{The Diagonal Term} \label{s.diagonal}  

The diagonal term is the hardest in many
local $Tb$ arguments, and this is  true also in our situation; the goal
is to prove the following inequality:
\begin{equation}\label{e.dia<}
	 	\Bigl\lvert 
	\sum_{P \in \mathcal G ^{1}} 
	\sum_{\substack{Q\in \mathcal G ^{2}\\ Q\cap 3P \neq \emptyset \,,\,  2 ^{-r} \ell P \le \ell Q \le \ell P }} 
	\langle  T \Delta _P f_1, \Delta _Q f_2 \rangle 
	\Bigr\rvert
	\lesssim \big\{C_r(1+ \mathbf T _{\textup{loc}}) + r \upsilon _1\mathbf T \big\} \lvert  Q ^{0}\rvert \,.  
\end{equation}
Note, in particular, that the bound in terms of $ \mathbf T $  has leading absorbing constant $r  \upsilon _1$. 
On the other hand, $ \mathbf T _{\textup{loc}}$ has a leading constant 
$ C_r$ that will be exponential in $ r$. The implied constant is independent
of the absorption parameters.

The first step in the proof is not so straight forward.
Its purpose is to avoid terms
$\{2^{cr} \upsilon_1\mathbf{T}\}\lvert Q^0\rvert$ that cannot be absorbed.
To explain, let us pass back to the heavier notation 
$ \Delta _P f_1 = \Delta _{P} ^{\beta ^{1}} f_1$;  the point of the estimate below is that we will replace 
$ \beta ^{1}$ in the twisted differences by $ b ^{1}$. 

%%%%%%%%%%%%%%%%%%%%%%%%%%%%%% LEMMA LEMMA LEMMA
\begin{lemma}\label{l.back2b} There holds 
\begin{equation*}
	\Bigl\lvert 
	\sum_{P \in \mathcal G ^{1}} 
	\sum_{\substack{Q\in \mathcal G ^{2}\\ Q\cap 3P \neq \emptyset \,,\,  2 ^{-r} \ell P \le \ell Q \le \ell P }} 
	\langle  T (\Delta _P ^{\beta ^{1}} f_1  - \Delta _P ^{b^1} f_1)  , \Delta _Q f_2\rangle 
	\Bigr\rvert
	\lesssim \{ r \upsilon _1 \mathbf T \} \lvert  Q ^{0}\rvert\,.
\end{equation*}
\end{lemma}
%%%%%%%%%%%%%%%%%%%%%%%%%%%%%% LEMMA LEMMA LEMMA

%%%%%%%%%%%%%%%%%%%%%%%%%%%%%% PROOF PROOF PROOF
\begin{proof}
	Perturbation inequality is the principal tool here.
	By introducing independent Rademacher variables $\{\epsilon_P\}_{P\in\mathcal{G}^1}$ that are jointly supported
	on a probability space $\Omega=\{-1,1\}^{\mathcal{G}^1}$, 
	we have, for integers $ 0\le s \le r$, 
	\begin{align*}
\Bigl\lvert 
	\sum_{P \in \mathcal G ^{1}} &
	\sum_{\substack{Q\in \mathcal G ^{2}\\ Q\cap 3P \neq \emptyset \,,\,  2 ^{-s} \ell P = \ell Q  }} 
	\langle  T (\Delta _P ^{\beta ^{1}} f_1  - \Delta _P ^{b^1} f_1)  , \Delta _Q  f_2\rangle 
	\Bigr\rvert
	\\
	& =
	\Big\lvert
	\int_\Omega
	\Big\langle 
	\sum_{P \in \mathcal G ^{1}} 
	\epsilon_P T (\Delta _P ^{\beta ^{1}} f_1  - \Delta _P ^{b^1} f_1),
	\sum_{R\in\mathcal{G}^1}
	\epsilon_R \sum_{\substack{Q\in \mathcal G ^{2}\\ Q\cap 3R \neq \emptyset \,,\,  2 ^{-s} \ell R = \ell Q  }} 
	\Delta _Q  f_2\Big\rangle d \epsilon
	\Bigr\rvert\\
	&\lesssim \bigg\{ \int_\Omega
	\Bigl\lVert 
	T \Bigl( 
		\sum_{P \in \mathcal G ^{1}} \epsilon_P\big\{\Delta _P ^{\beta ^{1}}  - \Delta _P ^{b^1} 
	 \big\} f_1\Bigr)
	 \Bigr\rVert_{p_1}^{p_1}\, d\epsilon \bigg\}^{1/p_1}
	\\&\qquad\qquad\qquad  \times \bigg\{ \int_\Omega
	 \Bigl\lVert  
	 \sum_{P \in \mathcal G ^{1}} 
	\sum_{\substack{Q\in \mathcal G ^{2}\\ Q\cap 3P \neq \emptyset \,,\,  2 ^{-s} \ell P = \ell Q  }} 
	\epsilon_P\Delta _Q f_2 
	\Bigr\rVert_{p_1'}^{p_1'}\,d\epsilon \bigg\}^{1/p_1'}\,.
		\end{align*}
Extract the operator norm from the first factor, and after that
apply
Khintchine's inequality and Theorem \ref{t.perturb_special}.
Theorem \ref{t.twisted} 
is used to estimate the second factor, but only
after having changed the order of summation and
having applied the H\"older's inequality and inequality
$p_1'\le p_2$. Finally, summing
the $s$-series yields the upper bound $r\upsilon _1\mathbf T   \lvert  Q ^{0}\rvert$.
\end{proof}
%%%%%%%%%%%%%%%%%%%%%%%%%%%%%% PROOF PROOF PROOF

It remains to prove  Lemma \ref{e.Dia} below.
Indeed, a straight forward application of 
inequality \eqref{e.Box<}, combined with the two lemmata  \ref{l.back2b} and \ref{e.Dia}, completes the proof 
of the diagonal estimate \eqref{e.dia<}. 

\begin{lemma}\label{e.Dia}
Assume that $3P\cap Q\not=\emptyset$,  and that $ 2 ^{-r} \ell P \le \ell Q \le \ell P$. 
Then 
\begin{align}
	\bigl\lvert \langle  T \Delta _P ^{b ^{1}}  f_1, \Delta _Q f_2\rangle\bigr\rvert 
	&\lesssim
	\{ 1+\mathbf{T}_{\textup{loc}}    \} \cdot \langle \Box_P ^{b^1} f_1 \rangle_P
	\langle \Box_Q f_2 \rangle_Q\lvert P\rvert  \,. 
	\end{align}
	\end{lemma}

\begin{proof}
The cube $ P$ has $ 2 ^{n}$ children $ P'$. If a child $ P'$ is not a stopping cube, 
$ \Delta _P^{b^1} f_1\cdot \mathbf{1}_{P'}$ is equal to a multiple of $b^{1} _{S_1}\mathbf{1}_{P'}$, 
where $ S_1$ is the $ \mathcal S ^{1}$ parent of $ P$, and the multiple is given by the value of the  half-twisted martingale difference 
$\widetilde D _P^{b^1} f_1$ on $ P'$.   If $ P'$ is a stopping cube, then 
$ \Delta_P^{b^1} f_1\cdot \mathbf{1}_{P'}$ in addition involves a bounded multiple 
of $ b^{1} _{P'}$.  
In both cases, the constant multiples are bounded
in absolute value by $\langle \Box_P^{b^1} f_1\rangle_P$,
compare to definition \eqref{e.Box}.
Similar comments apply to $ \Delta  _{Q} f_2=\Delta  _{Q}^{\beta^2} f_2$ restricted to a child $Q'$.  By these considerations, we need  to prove the estimate 
\begin{equation*}
	\lvert  \langle  T \psi ^{1} , \psi ^{2} \rangle\rvert \lesssim  	\{ 1+ \mathbf T _{\textup{loc}}\}
\lvert  P\rvert\,, 
\end{equation*}
where $ \psi ^{1}  = b^{1} _{S_1} \mathbf 1_{P'}$, and   $ \psi ^{2} \in
\{ \beta ^{2} _{S_2} \mathbf 1_{Q'}, \beta ^{2 } _{Q'}\mathbf{1}_{\{Q'\in\mathcal{S}^2\}}\}$, where $ S_2$ is the $ \mathcal S ^{2}$ parent of $ Q$.  A similar estimate is also required when $ \psi ^{1} =b^{1 } _{P'} $, on the condition that $ P'\in \mathcal S ^{1}$. 
An obstruction  is that, even though  the stopping conditions control the local norm of $ T b^{1} _{S_1}$,  
 we may have the restriction $b^{1} _{S_1} \mathbf 1_{P'}$ \emph{inside}  the operator $ T$.  

The case of $ \psi ^{1} =  b^{1 } _{P'} $, where we require that $P'\in\mathcal{S}^1$, 
is especially easy, since the obstruction just mentioned does not arise. 
By the construction of the stopping cubes, and the fact that 
Lebesgue measure is doubling, 
$	\lvert  \langle  T   b ^{1 } _{P'}  , \psi ^{2} \mathbf 1_{P'} \rangle\rvert \lesssim 
	\lVert \mathbf 1_{P'}T b^{1} _{P'}\rVert_{p_2'} \lVert\psi ^{2}   \rVert_{p_2}
	\lesssim 
	\mathbf  T _{ \textup{loc}} 
\lvert  P\rvert$;
here we  complied to the stopping rules by  restricting $\psi_2$ to  $ P'$.  
Concerning the contribution outside of  $P'$, 
inequality $p_1'\le p_2$ and the Hardy's inequality 
in Lemma \ref{l.hardy} together yield
$	\lvert \langle Tb^1_{P'}, \psi^2 \mathbf{1}_{Q'\setminus P'}\rangle\rvert
	\lesssim 
\lvert  P\rvert$.

In the case of $ \psi ^{1} =  b^{1} _{S_1} \mathbf 1_{P'}$ we must face the obstruction.  
Again, write 
$	\psi  ^{2} 
	= 
	\psi  ^{2}  \mathbf 1_{P'} +  \psi  ^{2}  \mathbf 1_{Q' \setminus P'}$. The Hardy's inequality controls the second term, giving 
$	\lvert \langle Tb^{1} _{S_1} \mathbf 1_{P'} ,  \psi  ^{2}   \mathbf 1_{Q' \setminus P'}\rangle \rvert
	\lesssim \lvert  P\rvert$.
For the first term, we return to the local $ Tb$ hypothesis, and write 
\begin{equation*}
	\psi ^{2} \mathbf 1_{P'} = \langle  \psi ^{2}   \rangle _{P'} b ^{2} _{P'} + 
	( \psi ^{2}\mathbf{1}_{P'} -  \langle  \psi ^{2} \rangle _{P'} b ^{2} _{P'})
	=:  \langle  \psi ^{2} \rangle _{P'} b ^{2} _{P'} +   \widetilde \psi  ^{2} \,. 
\end{equation*}
 The advantage of 
the first summand on the right is that the local $ Tb$ hypothesis gives us 
\begin{align*}
	\lvert \langle  \psi ^{2} \rangle _{P'}  \cdot  	\langle T b^{1} _{S_1} \mathbf 1_{P'} ,  b ^{2} _{P'} \rangle  \rvert 
	&= 
	\lvert \langle  \psi ^{2} \rangle _{P'}  \cdot  	\langle b^{1} _{S_1} \mathbf 1_{P'} , T ^{\ast}  b ^{2} _{P'} \rangle  \rvert 
\lesssim 
\mathbf T _{\textup{loc}}  \lvert  \langle  \psi ^{2} \rangle _{P'} \rvert  \cdot  \lvert  P\rvert
 \lesssim 
\mathbf T _{\textup{loc}}  \lvert  P\rvert\,. 
\end{align*}
	The advantage of the second summand is that it has integral zero: $ \int _{P'}  \widetilde \psi  ^{2} \; dx =0 $.  
	Note that also $ \lVert \widetilde \psi ^{2}\rVert_{p_2} \lesssim \lvert  P\rvert ^{1/p_2} $.  Take $ \mathcal P$ to be the 
cubes of the form $ P' \dot+ u$, where $ u \in \{-1,0,1\} ^{n} \setminus \{(0,0,\ldots)\}$.  Then, 
\begin{align*}
\lvert  	\langle Tb^1_{S_1} \mathbf 1_{P'},\widetilde\psi ^{2}  \rangle  \rvert 
&\le 
\lvert  	\langle T b^1_{S_1} ,\widetilde\psi ^{2}  \rangle  \rvert + 
\lvert  	\langle T b^1_{S_1} \mathbf 1_{S_1 \setminus P'},\widetilde\psi ^{2}  \rangle  \rvert 
\\
& \le 
\lvert  	\langle T b^1_{S_1} ,\widetilde\psi ^{2}  \rangle  \rvert +  
\sum_{R\in \mathcal P} \lvert  	\langle T b^1_{S_1} \mathbf 1_{R} ,\widetilde\psi ^{2}  \rangle  \rvert 
+ \lvert  	\langle T b^1_{S_1} \mathbf 1_{S_1 \setminus 3P'} , \widetilde \psi  ^{2}  \rangle  \rvert\,.
\end{align*}  
The first term is controlled by the stopping rules: 
$	\lvert  	\langle T b^1_{S_1} ,\widetilde\psi ^{2}  \rangle  \rvert \lesssim \mathbf T _{\textup{loc}}
\lvert  P\rvert$.
The second sum is finite, and each summand is precisely of the type that appears in the Hardy's Inequality.
Indeed, although $R\in\mathcal{P}$ need not be contained in $P$, by \eqref{e.inf}
we nevertheless have 
\begin{equation*}
	\langle \lvert b_{S_1}^1\rvert^{p_2'}\rangle_R^{1/p_2'} \le 
	\langle \lvert  b^{1} _{S_1}\rvert ^{p_1}  \rangle _{R}^{1/p_1}
	\le \big\{4 ^{n} \inf _{x\in P} M  \lvert  b^{1} _{S_1}\rvert ^{p_1}\big\}^{1/p_1}\lesssim 1\,.
\end{equation*}
And,  it follows that 
\begin{equation*}
\sum_{R\in \mathcal P} \lvert  	\langle T b^1_{S_1} \mathbf 1_{R} ,\widetilde\psi ^{2} \rangle  \rvert  
\lesssim  \sum_{R\in\mathcal{P}}\lVert b_{S_1}^1 \mathbf{1}_R\rVert_{p_2'}
\cdot \lvert P\rvert^{1/p_2} \lesssim 
 \lvert  P\rvert \,.  
\end{equation*}
Finally, by a similar estimate as in (the proof of) Lemma~\ref{l.basic_far}, and the stopping rules, 
\begin{equation*}
\lvert  	\langle T b^1_{S_1} \mathbf 1_{S_1 \setminus 3P'} , \widetilde \psi  ^{2}  \rangle  \rvert 
\lesssim \inf _{x\in P'} M b^{1} _{S_1}  \int \lvert  \widetilde \psi  ^2 \rvert\; dx  \lesssim \lvert  P\rvert  \,.  
\end{equation*}
This completes the proof of Lemma \ref{e.Dia}. 
\end{proof}

\end{document}